\documentclass[11pt]{article}

\usepackage[margin=1in]{geometry}
\usepackage{amsthm,amsmath,amsfonts,amssymb}
\usepackage[authoryear]{natbib}
\usepackage[colorlinks,citecolor=blue,urlcolor=blue]{hyperref}
\usepackage{graphicx}
\usepackage{booktabs}

\numberwithin{equation}{section}
\theoremstyle{plain}

\newtheorem{lem}{Lemma}
\newtheorem{cor}{Corollary}
\newtheorem{prp}{Proposition}
\newtheorem{thm}{Theorem}
\theoremstyle{definition}
\newtheorem{rmk}{Remark}
\newtheorem{exm}{Example}
\newtheorem{dfn}{Definition}

\newtheorem{examplex}{Example}
\newenvironment{examplecont}[2]
  {\renewcommand{\theexamplex}{\ref{#1} (Part #2)}\begin{examplex}}
  {\end{examplex}}
\newcommand{\stitle}[1]{}
\newcommand{\sdescription}[1]{}

\title{Measuring Evidence against Exchangeability and Group Invariance with E-values}
\author{Nick W. Koning\\Econometric Institute, Erasmus University Rotterdam\\n.w.koning@ese.eur.nl}
\date{}

\begin{document}
\maketitle

\begin{abstract}
	We study e-values for quantifying evidence against exchangeability and general invariance of a random variable under a compact group.
	We start by characterizing such e-values, and explaining how they nest traditional group invariance tests as a special case.
	We show they can be easily designed for an arbitrary test statistic, and computed through Monte Carlo sampling.
	We prove a result that characterizes optimal e-values for group invariance against optimality targets that satisfy a mild orbit-wise decomposition property.
	We apply this to design expected-utility-optimal e-values for group invariance, which include both Neyman--Pearson-optimal tests and log-optimal e-values.
	Moreover, we generalize the notion of rank- and sign-based testing to compact groups, by using a representative inversion kernel.
	In addition, we characterize e-processes for group invariance for arbitrary filtrations, and provide tools to construct them.
	We also describe test martingales under a natural filtration, which are simpler to construct.
	Peeking beyond compact groups, we encounter e-values and e-processes based on ergodic theorems.
	These nest e-processes based on de Finetti's theorem for testing exchangeability.
\end{abstract}


	\section{Introduction}\label{sec:intro}
		Testing group invariance is an old and fundamental problem in hypothesis testing.
		It covers many non-parametric tests, including permutation tests, conformal inference, various popular multiple testing methods, many causal inference methods, and even the $t$-test.
		Tests for group invariance are attractive and widely used, as invariance properties under a null hypothesis are often easy to defend.
		Rejecting this invariance then also rejects the null hypothesis of interest.
		
		Up to the present, only a limited class of group invariance tests has been explored, of a form inspired by the traditional Neyman--Pearson framework of testing. 
		The primary contribution of this manuscript is to go beyond this traditional framework by measuring evidence against group invariance with e-values \citep{shafer2021testing,vovk2021values,howard2021time,ramdas2023gametheoretic,grunwald2023safe, koning2024continuous}.
		
		Before we detail our contributions, we first briefly discuss the traditional permutation test, which serves as a prototypical example of a traditional group invariance test.
		We follow this with a short primer on e-values.
		
		\subsection{Traditional permutation test}
			Consider the random variable $X^n = (X_1, \dots, X_n)$, $n \geq 1$, and suppose we are interested in testing whether it is exchangeable:
			\begin{align*}
				H_0 : X^n \text{ is exchangeable}.
			\end{align*}
			Here, exchangeability means that $X^n$ is equal in distribution to every permutation $PX^n$ of its elements.
			As an example, $X^n$ is exchangeable if its components $X_1, \dots, X_n$ are i.i.d., but there also exist non-i.i.d. exchangeable distributions.
			
			Given some test statistic $T$, a traditional `permutation p-value' is given by
			\begin{align*}
				p_n
					= \mathbb{P}_{\overline{P}_n}\left(T(\overline{P}_nX^n) \geq T(X^n)\right),
			\end{align*}
			where $\overline{P}_n \sim \textnormal{Unif}(\mathfrak{P}_n)$ is uniformly distributed on the permutations $\mathfrak{P}_n$ of $n$ elements.
			This p-value can be understood as the proportion of test statistics calculated from the rearranged (`permuted') data that exceed or match the original test statistic.
			
			Instead of a p-value, we can equivalently formulate a traditional permutation test
			\begin{align*}
				\varepsilon_n^{(\alpha)}
					= \mathbb{I}\{p_n \leq \alpha\} / \alpha,
			\end{align*}
			where we follow \citet{koning2024continuous} in defining a level $\alpha$ test $\varepsilon_n^{(\alpha)}$ as $[0, 1/\alpha]$-valued, so that $\varepsilon_n^{(\alpha)} = 1/\alpha$ indicates a rejection at level $\alpha$ and $\varepsilon_n^{(\alpha)} = 0$ a non-rejection.
			
			Permutation tests and p-values are well-known to be valid in finite samples:
			\begin{align}\label{eq:type-1}
				\sup_{n, \alpha} \mathbb{E}^{\mathbb{P}}[\varepsilon_n^{(\alpha)}] 
					\equiv \sup_{n, \alpha} \mathbb{E}^{\mathbb{P}}[\mathbb{I}\{p_n \leq \alpha\} / \alpha] \leq 1,
			\end{align}
			for every exchangeable distribution $\mathbb{P}$, or, equivalently, $\mathbb{P}(p_n \leq \alpha) \leq \alpha$, for all $n$ and $\alpha$.
			In fact, this can even be made to hold with equality by breaking ties through randomization.
			
			Permutation tests are a special case of more general group invariance tests, which are obtained by simply replacing the group of permutations $\mathfrak{P}_n$ by some other compact group $\mathcal{G}$ that acts on our sample space, interpreting $\textnormal{Unif}(\mathcal{G})$ as the Haar measure. 
			In Section \ref{sec:background}, we cover the necessary background on compact groups and their actions on sample spaces.
			
		\subsection{Primer on e-values}
			Traditional group invariance tests are \emph{binary} by construction: such a test either rejects the hypothesis at level $\alpha$ or not: $\varepsilon^{(\alpha)} = 1/\alpha$ or $\varepsilon^{(\alpha)} = 0$.
			Recently, there has been much interest in moving away from such binary tests towards `fuzzy' tests that exploit the entire interval $[0, 1/\alpha]$ or even $[0, \infty]$ \citep{koning2024continuous}.
			These tests have been popularized under the name \emph{e-values}.
			Such e-values are to be used as a continuous measure of evidence against the hypothesis, where their value in $[0, 1/\alpha]$ or $[0, \infty]$ is \emph{directly} interpreted as evidence.
			The introduction of the e-value has led to a series of breakthroughs in sequential testing, multiple testing and even in single-hypothesis testing.
			
			For e-values bounded to $[0, 1/\alpha]$, the Neyman--Pearson lemma tells us that binary e-values (tests) automatically arise when maximizing the power $\mathbb{E}^{\mathbb{Q}}[\varepsilon]$ for simple null hypotheses, and also when testing group invariance \citep{lehmann1949theory}.
			To obtain non-binary e-values, we must consider other power-targets, such as $\mathbb{E}^{\mathbb{Q}}[\log \varepsilon]$.
			E-values that maximize this target have also been dubbed `log-optimal', `GRO' or `numeraire' \citep{shafer2021testing, grunwald2023safe, larsson2024numeraire}.
			While this log-target is most popular, we may also consider more general utilities \citep{koning2024continuous}.
			
			A sequential generalization of an e-value is an e-process $(\varepsilon_n)_{n \geq 1}$.
			Such an e-process is said to be \emph{anytime valid} if \eqref{eq:type-1} holds not just for data-independent $n$, but for stopping times:
			\begin{align*}
				\mathbb{E}^{\mathbb{P}}[\varepsilon_\tau] \leq 1,
			\end{align*}
			for every stopping time $\tau$, and every group invariant distribution $\mathbb{P}$.
			Such e-processes relieve us from pre-specifying a number of observations, and permit us to continuously monitor the data and current evidence, and stop whenever we desire.
			
			Another interpretation of the e-value may be found in its relationship to the p-value.
			Specifically, the reciprocal $\mathfrak{p} = 1/\varepsilon$ of an e-value $\varepsilon$ is a special kind of p-value that is valid under a much stronger Type-I error guarantee \citep{koning2023markov, grunwald2023beyond}, also called the post-hoc level Type-I error:
			\begin{align*}
				 	\mathbb{E} \left[\sup_{\alpha} \mathbb{I}\{\mathfrak{p} \leq \alpha\} / \alpha\right]
					\equiv \mathbb{E}\left[1/\mathfrak{p}\right]
					\leq 1.
			\end{align*}
			This is stronger than the traditional Type-I error \eqref{eq:type-1}, as the supremum over $\alpha$ is now inside the expectation which means that it is also valid when using data-dependent significance levels $\alpha$.
			Indeed, $\alpha$ may be chosen \emph{post-hoc}.
			
			The smallest data-dependent level at which we reject is the p-value $\mathfrak{p}$ itself.
			As a consequence, we may truly `reject at level $\mathfrak{p}$' and still retain a generalized Type-I error control for data-dependent significance levels \citep{koning2023markov}.		
			This is certainly not permitted for traditionally valid p-values, which only offer a guarantee when compared to an independently specified level \eqref{eq:type-1}.
			To distinguish these p-values from traditional p-values, they are sometimes referred to as \emph{post-hoc p-values}.
			Such a post-hoc p-value, and by extension the e-value, may therefore be viewed as a p-value that offers a generalized Type-I error guarantee, even when interpreted continuously.
						
		\subsection{First contribution: characterizing and computing e-values for group invariance}
			In Section \ref{sec:ph_gi}, we study the characterization of e-values for group invariance for a compact group $\mathcal{G}$.
			There, we show that $\varepsilon$ is a valid e-value for $\mathcal{G}$ invariance if and only if
			\begin{align*}
				\mathbb{E}_{\overline{G}}[\varepsilon(\overline{G} x)] \leq 1,
			\end{align*}
			where $\overline{G} \sim \textnormal{Unif}(\mathcal{G})$, for every $x$ in the sample space.
			That is, it needs to be valid under a uniform distribution, $\overline{G}x \sim \textnormal{Unif}(O_{x})$, on each `orbit' $O_{x} = \{Gx : G \in \mathcal{G}\}$.
			In fact, as the data $X$ identifies the orbit in which it falls, we find that the e-value only needs to be valid for the uniform distribution $\textnormal{Unif}(O_{X})$ conditionally on the orbit $O_{X}$ in which the data $X$ lands.
			
			We use this to show that an e-value is exactly valid if and only if the e-value is of the form 
			\begin{align*}
				\varepsilon_T(x) = \frac{T(x)}{\mathbb{E}_{\overline{G}}[T(\overline{G} x)]},
			\end{align*}
			for some non-negative and appropriately integrable test statistic $T$.
			We find e-values of this form can be computed easily by replacing the denominator with a Monte Carlo average over i.i.d. samples from $\overline{G}$.
			Indeed, we show such a `Monte Carlo e-value' is valid in expectation over the Monte Carlo draws.
			As a side contribution, we also derive a weaker condition under which traditional group invariance tests are valid.
		
		\subsection{Second contribution: optimal e-values for group invariance}
			In Section \ref{sec:optimal}, we consider optimal e-values for group invariance against an abstract power-target.
			Under a mild assumption that the target monotonically aggregates local optimality targets on orbits, we find that an e-value is optimal for group invariance if it is locally optimal on each orbit.
			In fact, such e-values are optimal uniformly in any monotone aggregation of the local orbit-wise targets.
			We apply this idea to derive expected-utility-optimal e-values for group invariance, which optimize $\mathbb{E}^{\mathbb{Q}}[U(\varepsilon)]$ for some alternative $\mathbb{Q}$ and utility function $U$.
			
			To present such optimal e-values, it is helpful to introduce the $\mathcal{G}$-average $\overline{\mathbb{Q}}$ of the alternative $\mathbb{Q}$, which may be constructed by averaging the $\mathbb{Q}$-probability mass of each event over group-transformations of the event.
			Under some regularity conditions on $U$, an expected-utility-optimal e-value is given by
			\begin{align*}
				\varepsilon^U = (U')^{-1}\left(\lambda^* \Big/ \frac{d\mathbb{Q}}{d\overline{\mathbb{Q}}}\right),
			\end{align*}
			for some orbit-dependent constant $\lambda^*$, where we assume $\mathbb{Q} \ll \overline{\mathbb{Q}}$ here for simplicity.
			
			Specializing this to log-utility $U : x \mapsto \log(x)$ yields $\varepsilon^{\textnormal{log}} = d\mathbb{Q} / d\overline{\mathbb{Q}}$.
			This simultaneously reveals that $\overline{\mathbb{Q}}$ may be viewed as the Reverse Information Projection (`RIPr') of $\mathbb{Q}$ onto the collection of $\mathcal{G}$ invariant probabilities under the KL-divergence \citep{grunwald2023safe, lardy2024reverse, larsson2024numeraire}.
			
			The `Neyman--Pearson utility function' $U_\alpha : x \mapsto x \wedge 1/\alpha$ \citep{koning2024continuous} coincides with the classical notion of `power', for which an optimal e-value is $\varepsilon_\alpha^{\textnormal{NP}} = \frac{1}{\alpha}\mathbb{I}\left\{\frac{d\mathbb{Q}}{d\overline{\mathbb{Q}}} > c_\alpha\right\} 
					+ \frac{k}{\alpha}\mathbb{I}\left\{\frac{d\mathbb{Q}}{d\overline{\mathbb{Q}}} = c_\alpha\right\}$,
			where $c_\alpha$ and $k$ are certain orbit-dependent constants.
			This recovers the Neyman--Pearson optimal test for group invariance derived by \citet{lehmann1949theory}.
			
			Instead of specifying an alternative on the original sample space, we also consider specifying an alternative on each orbit, as well as conditional on the data-orbit.
			Finally, we also consider specifying an alternative on the group, by passing the data through a unique representative inversion function, which maps it to the group.
			We show that this nests traditional rank and sign-based tests, by using the fact that the ranks and signs are in bijection with the group of permutations and sign-flips.
			
		\subsection{Third contribution: e-processes and test martingales for group invariance}
			In Section \ref{sec:e-processes}, we study e-processes for group invariance.
			We generalize the characterization of e-values to e-processes for arbitrary filtrations, and show how they may be constructed by tracking an infimum over martingales for uniform distributions $\textnormal{Unif}(O)$ on orbits $O$.
			We link this to our derivation of optimal e-values, by showing how optimal e-values induce such orbit-wise martingales.
			Moreover, we identify a key challenge when constructing e-processes: the orbit in which the data lands is not necessarily measurable at the start of the filtration, and is possibly not even measurable for any sigma-algebra in the filtration.
			This explains why we cannot reduce the sequential problem to testing uniformity on the orbit in which the data lands: we must instead keep track of all orbits in which the data may feasibly land.
			
			In addition, we peek beyond compact groups by relying on an ergodic theorem for possibly non-compact groups.
			Here, we find e-values and e-processes may be characterized by an infimum over e-values for ergodic measures, which replace the role of uniform distributions on orbits in the compact setting.
			An example of an ergodic theorem is de Finetti's theorem, where these ergodic measures are the i.i.d. probabilities, which was explored in the context of e-processes for binary and $d$-ary data by \citep{ramdas2022testing}.
			
			In Section \ref{sec:seq_invariance}, we zoom-in towards a particular setup in which we consider testing whether the sequence $(X^n)_{n \geq 1}$ is invariant under a sequence of groups $(\mathcal{G}_n)_{n \geq 1}$.
			We study test martingales for this setup, which are e-processes that are not just anytime valid but satisfy a stronger condition.
			We show how they may easily be constructed by tailoring conditional e-values for their increments based on subgroups that stabilize the previous data.
			
			In Appendix \ref{sec:impoverishing}, we study the impoverishment of filtrations, which is the process of deliberately moving to a less-informative filtration, usually in order to design a more powerful e-process by shrinking the set of stopping times under which it must be valid.
			In practice, this means we restrict ourselves to looking at some statistic of the data, instead of at the underlying data itself.
			We show how we may generally impoverish filtrations in the context of group invariance, by using a statistic that is equivariant under a subgroup of the group.
			The problem then reduces to testing invariance of the statistic under the subgroup. 
			
			We then focus on equivariant statistics for which the subgroup only has a single orbit on its codomain.
			We find that an example of such an equivariant statistic is the unique representative inversion that maps the data to the group.
			This nests the idea studied by \citet{vovk2021testing} to pass to the ranks of the data in the context of testing exchangeability, which we show may be viewed as the representative inversion for exchangeability. 
			 
		\subsection{Simulations, application and illustration}
					
			We empirically illustrate our methods in two simulation experiments and an application.
			The first simulation mimics a standard case-control experiment under random treatment allocation.
			In the second experiment, we compare a sign-flipping e-process to one based on \citet{de1999general}, and find that ours is dramatically more powerful.
			The application is to the `hot hand' phenomenon in basketball, which is the belief that hitting a basketball shot increases the chances of hitting subsequent shots \citep{gilovich1985hot}.
			This is frequently studied by assuming that the shot outcomes are exchangeable in absence of the hot hand, so that rejecting exchangeability also rejects the hot hand \citep{miller2018surprised}.
			We leverage the powerful merging properties of e-values by multiplying e-values across players to obtain a more powerful e-value.
			This merging of evidence is highly relevant for the hot hand, as a single shot sequence of a player is known to contain little evidence regarding the hot hand \citep{ritzwoller2022uncertainty}.			
			
			In Appendix \ref{sec:gaussian}, we illustrate these methods on the problem of testing invariance on $\mathbb{R}^d$, $d \geq 1$, under an arbitrary compact group of orthonormal matrices, against a simple alternative that is a location shift under normality.
			For the special case of spherical invariance, this is connected to an example from \cite{lehmann1949theory} regarding the optimality of the $t$-test, which we slightly generalize.
			We also consider sign-symmetry, which produces an e-value that can be viewed as an admissible version of an e-value based on \cite{de1999general}.
			Furthermore, we consider exchangeability where we discover an interesting link to the softmax function.

		\subsection{Related literature}\label{sec:lit}
			At first glance, our work may seem intimately related to the work of \cite{perez2022statistics}.
			However, they consider invariance of \emph{collections of distributions} (both the null and the alternative), whereas we consider invariance of \emph{distributions themselves}.
			Specifically, a collection of distributions $\mathcal{P}$ is said to be invariant under a transformation $G$ if for any $\mathbb{P} \in \mathcal{P}$, its transformation $G\mathbb{P}$ by $G$ is also in $\mathcal{P}$.
			In contrast, invariance of a distribution $\mathbb{P}$ means that its transformation $G\mathbb{P}$ is equal to $\mathbb{P}$ itself.
			Intuitively, their work can be interpreted as \emph{testing in the presence of an invariant model}, whereas we consider \emph{testing whether the data generating process is invariant}.
						
			As our null hypothesis consists exclusively of invariant distributions, it is technically also invariant, so that one may believe their results may still apply under appropriate assumptions on the alternative.
			However, \cite{perez2022statistics} require that the group action is free, which means that if $G \mathbb{P} = \mathbb{P}$ for some $\mathbb{P} \in \mathcal{P}$ then $G$ must be the identity.
			In other words, applying a non-identity transformation to $\mathbb{P}$ \emph{must} change it.
			Our form of invariance instead requires that $G\mathbb{P} = \mathbb{P}$ for every $G$ and $\mathbb{P} \in \mathcal{P}$.
			This means that the settings do not overlap, except for the uninteresting setting in which the group only contains the identity element.
						
			\cite{vovk2023testing} independently derives a permutation e-value for testing exchangeability of binary random variables against a single specific alternative hypothesis.			
			Moreover, \citet{vovk2021testing} explores testing exchangeability in a sequential setup, by passing from the original data to its ranks.
			He exploits the fact that the sequential ranks are independent from the past ranks under exchangeability.
			He then converts these ranks into independent e-values, which are multiplied together to construct a test martingale under the rank-filtration.
			\citet{lardy2024anytime} apply this to testing group invariance in a setting similar to Section \ref{sec:seq_invariance} and Appendix \ref{sec:impoverishing}.
			We show how rank-based approaches may be viewed as a special case of using a unique representative inversion, where ranks appear as the representative inversion for exchangeability.
			
			A link between the softmax function and e-values for exchangeability was also made in unpublished early manuscripts of \cite{wang2022false} and \cite{ignatiadis2023values}, which they call a `soft-rank' e-value.
			In Remark \ref{rmk:softrank} in Appendix \ref{sec:softmax}, we explore the connection to our softmax likelihood ratio statistic, and find that their soft-rank e-value can be interpreted as a more volatile version.
			
			Testing the symmetry of a distribution, which we touch in Appendix \ref{sec:symmetry}, was also studied by \citet{ramdas2022admissible}, \citet{vovk2024nonparametric} and \citet{larsson2024numeraire}.
			
		\subsection{Notation and underlying assumptions}
			Throughout the paper, every `space' we consider is assumed to be second-countable locally compact Hausdorff, equipped with a Borel $\sigma$-algebra.
			We intentionally suppress this topology and the $\sigma$-algebra whenever possible, for notational conciseness.
			To avoid ambiguity, we sometimes write expectations $\mathbb{E}$ with a superscript and/or subscript $\mathbb{E}_X^{\mathbb{P}}$ to make explicit the measure over which is being integrated ($\mathbb{P}$), and the random variables over which the integration takes place ($X$).
			We use similar subscripts for probabilities.
			We frequently use $\overline{G}$ to denote a Haar-distributed random element in $\mathcal{G}$, assumed to be independent from everything else unless stated otherwise.

	\section{Background: group invariance}\label{sec:background}
		In this section, we discuss all the necessary background on group invariance.
		We recommend \citet{eaton1989group} for deeper treatment of invariance in statistics.
		
		\subsection{Compact groups}\label{sec:gi}
			A group $\mathcal{G}$ is a set equipped with some associative binary operator `$\times$' that is closed under composition and inversion, and contains an identity element $I$.
			For brevity, we use juxtaposition $G_1G_2 = G_1 \times G_2$ to denote the binary operation, $G_1, G_2 \in \mathcal{G}$.
			A subset of a group that is also a group is called a \emph{subgroup}.
		
			Throughout, unless stated otherwise, all groups we consider are \emph{compact groups}.
			Compact groups are groups that are also compact topological spaces.
			Compact groups are special in that they admit a unique \emph{invariant probability measure} $\mathbb{P}_{\text{Haar}}$ on $\mathcal{G}$ called the \emph{Haar probability measure}, which satisfies
			\begin{align*}
				\mathbb{P}_{\text{Haar}}(GA) = \mathbb{P}_{\text{Haar}}(A), \text{ for all } G \in \mathcal{G},
			\end{align*}
			for every event $A$ on the group, where $GA := \{Ga : a \in A\}$ is the event $A$ translated by $G$.
			
			The Haar probability measure can be interpreted as the uniform probability measure on the group: whenever we shift an event $A$ by some element $G$ of the group, its probability remains unchanged.
			We use $\overline{G}$ to denote a Haar-distributed random variable on $\mathcal{G}$, and we also write $\textnormal{Unif}(\mathcal{G}) := \mathbb{P}_{\text{Haar}}$.
			
			\begin{exm}[Orthonormal matrices]\label{exm:rotation}
				A typical example of a compact group acting on a sample space is the collection of all $n \times n$ orthonormal matrices, $n \geq 1$, which acts on $\mathbb{R}^n$ through matrix multiplication.
				This group action rotates or flips $n$-vectors about the origin.
				Here, the identity element $I$ is the identity matrix.
				Moreover, the inverse of an orthonormal matrix is simply its transpose: $G^{-1} = G'$, which is also orthonormal.
				The Haar measure is the uniform distribution over orthonormal matrices, and $\overline{G}$ is an orthonormal matrix drawn uniformly at random.
			\end{exm}
			
		\subsection{Group actions and orbits}
			In statistics, we are often interested in the action of a group $\mathcal{G}$ on a sample space $\mathcal{Y}$.
			We also denote such a \emph{group action} through juxtaposition: $(G, y) \mapsto Gy$, and assume that it is continuous.
			
			A group action partitions the sample space into \emph{orbits}.
			The orbit of a sample point $y \in \mathcal{Y}$, denoted by $O_y = \{z \in \mathcal{Y}\ |\ z = Gy,\ G \in \mathcal{G}\}$, can be interpreted as the set of all sample points that can be reached when starting from $y$ and applying an element of the group to it.
			We use $\mathcal{Y} / \mathcal{G}$ to denote the collection of orbits.
			We assign a single point $[y]$ in each orbit as the \emph{orbit representative} of the orbit $O_y$.
			That is, $[y] = Gy$ for some $G \in \mathcal{G}$.
			This means $O_{y} = O_{[y]}$ for any $y \in \mathcal{Y}$.
			We use $[\mathcal{Y}]$ to denote the collection of orbit representatives, and we call the map $[\cdot] : \mathcal{Y} \to [\mathcal{Y}]$ that maps $y$ to its orbit representative an \emph{orbit selector}.
			We assume the orbit selector is chosen to be measurable, which is possible if $\mathcal{G}$ is compact.
			
			\begin{examplecont}{exm:rotation}{B}\label{exm:rotation:B}
				The group of orthonormal matrices acts on the sample space $\mathcal{Y} = \mathbb{R}^n$ through matrix multiplication.
				Here, the collection of orbits $\mathcal{Y} / \mathcal{G}$ is the collection of hyperspheres in dimension $n$, each with a different radius.
				Given a unit vector $\iota$, we can assign the vector $r\iota$ as the orbit representative of the orbit with radius $r \geq 0$.
				The corresponding orbit selector is the map $y \mapsto \|y\|_2\iota$, since $y$ is on the hypersphere with radius $\|y\|_2$.
			\end{examplecont}
					
			\begin{exm}[One orbit]\label{exm:1orbit}
				If there is just a single orbit, we say that the group acts \emph{transitively} on the sample space.
				As the orbits partition the sample space, this means the sample space is the orbit.
				
				This happens, for example, if our sample space \emph{is} our group: $\mathcal{Y} = \mathcal{G}$.
				Another example is if we take Example \ref{exm:rotation:B}, but replace the sample space $\mathcal{Y} = \mathbb{R}^n$ with some hypersphere in dimension $n$.
 				In Section \ref{sec:inversion_kernels}, we discuss how we may generally reduce to a single orbit, which turns out to be a useful tool in statistical contexts.
 			\end{exm}
 					
		\subsection{Group invariant probability measures}
			In statistics, a sample space comes equipped with a collection of probability measures $\mathcal{P}$.
			The group action of $\mathcal{G}$ on $\mathcal{Y}$ \emph{induces a group action on the set of probabilities} $\mathcal{P}$.
			We can define this group action $(G, \mathbb{P}) \mapsto G\mathbb{P}$ as mapping the probability measure $\mathbb{P}$ to a probability measure $G\mathbb{P}$ that returns the $\mathbb{P}$-probability of the translation $G^{-1}A$ of an event $A$:
			\begin{align*}
				G\mathbb{P}(A) := \mathbb{P}(G^{-1}A).
			\end{align*}
			
			We may then extend the idea of a Haar probability measure, which is an invariant probability on the group $\mathcal{G}$, to invariant probabilities on a sample space $\mathcal{Y}$: we say that a probability $\mathbb{P}$ is invariant if the group action does not affect the probability.
			
			\begin{dfn}[Invariant probability]\label{dfn:1}
				$\mathbb{P}$ is an \emph{invariant probability measure} if $G\mathbb{P} = \mathbb{P}$, for every $G \in \mathcal{G}$.
			\end{dfn}
						
			On each orbit $O$, there exists a unique invariant probability measure, which may therefore be safely called `the' uniform probability $\textnormal{Unif}(O)$ on the orbit $O$.
			If there is just a single orbit, this means there is a single invariant probability.
			This happens, for example, if $\mathcal{Y} = \mathcal{G}$: the uniform probability is then the Haar measure $\textnormal{Unif}(\mathcal{G})$.
			If there are multiple orbits, there are generally multiple invariant probabilities: any probability mixture over uniform distributions on orbits is an invariant probability.
			In fact, in Lemma \ref{lem:invariance_random_variable}, we see that the converse also holds: any invariant probability may be viewed as a mixture over uniform probabilities on orbits.

 			\begin{examplecont}{exm:rotation}{C}
 				A typical example of an invariant distribution on $\mathbb{R}^n$ under the group of orthonormal matrices is the $n$-dimensional i.i.d. Gaussian distribution with mean zero and some variance $\sigma^2 \geq 0$.
 				In fact, this almost characterizes the Gaussian: the multivariate standard Gaussians are the \emph{only} rotationally-invariant distributions that have independent marginals, by the Herschel-Maxwell theorem.
 			\end{examplecont}

		\subsection{Equivalent characterizations of invariance}
			Beyond Definition \ref{dfn:1}, there exist several equivalent ways to characterize an invariant probability measure.
			In Lemma \ref{lem:invariance_random_variable} we list a number of such equivalent definitions, expressed in terms of a random variable $Y$.
			These can also be expressed in terms of probability measures, if desired.
			But we find that discussing our results in the context of random variables generally yields more easily interpretable statements.
			A proof of these statements may be found in Chapter 4 of \citet{eaton1989group}.
			
			\begin{lem}[Equivalent definitions of invariance]\label{lem:invariance_random_variable}
				$Y$ is an \emph{invariant random variable} under a compact group $\mathcal{G}$ if one of the following equivalent conditions holds:
				\begin{enumerate}
					\item[1.] The law $\mathbb{P}_Y$ of $Y$ is invariant,
					\item[2.] $Y \overset{d}{=} GY$, for every $G \in \mathcal{G}$,
					\item[3.] $Y \overset{d}{=} \overline{G}Y$, where $\overline{G} \sim \textnormal{Unif}(\mathcal{G})$ independently,
					\item[4.] $Y \overset{d}{=} \overline{G}[Y]$, where $\overline{G} \sim \textnormal{Unif}(\mathcal{G})$ independently,
					\item[5.] $\textnormal{Unif}(O_{Y})$ is a version of the conditional law $\mathbb{P}_Y(\cdot \mid O_{Y})$ of $Y$ given $O_{Y}$.
				\end{enumerate}
			\end{lem}
			
			Condition 4 states that a draw from an invariant random variable $Y$ can be decomposed (deconvolved) into first sampling an orbit representative $[Y]$ and subsequently multiplying it by $\overline{G}$, independently sampled uniformly from $\mathcal{G}$.
			
			Condition 5 restates this in terms of orbits: a draw from $Y$ can be viewed as first sampling an orbit $O_{Y}$ using some unspecified process, and subsequently sampling uniformly from this orbit.
			This decomposition is the key to testing invariance, where the idea is to effectively discard the first part of this sampling process, and only test whether $Y$ is uniform conditional on the orbit in which it is observed.
			
			A useful property is $G\overline{G} \overset{d}{=} \overline{G}G \overset{d}{=} \overline{G}$, which follows from Lemma \ref{lem:invariance_random_variable} by considering the invariant random variable $Y = \overline{G}$, and the fact that the Haar measure on a compact group is both left- and right-invariant.
			
			\begin{examplecont}{exm:rotation}{C}
 				A draw from an $n$-dimensional standard Gaussian $Y$ can be decomposed into first sampling a radius (and so orbit) from a $\chi_n$-distribution, and subsequently independently drawing an $n$-vector uniformly from the hypersphere with this radius (from the sampled orbit).
 				Lemma \ref{lem:invariance_random_variable} states that any rotationally invariant random variable can be characterized as such: first sampling a radius using some distribution, and subsequently independently drawing an $n$-vector uniformly from the sampled orbit.
 			\end{examplecont}
 			
 		\subsection{Constructing invariant probability measures and random variables}\label{sec:group_averaged_measure}
 			We can construct an invariant probability from a probability measure $\mathbb{P}$ by averaging it over the group:
			\begin{align*}
				\overline{\mathbb{P}} := \mathbb{E}_{\overline{G}}[\overline{G}\mathbb{P}],
			\end{align*}
			where $\overline{G} \sim \textnormal{Unif}(\mathcal{G})$.
			These \emph{group-averaged invariant measures} play a central role in optimal tests and e-values.
			For an invariant probability $\mathbb{P}$, we have $\mathbb{P} = \overline{\mathbb{P}}$, so that this averaging has no effect.
			We can also express this in terms of random variables: if $Y \sim \mathbb{P}$, then $\overline{G}Y \sim \overline{\mathbb{P}}$.

		\subsection{Invariance through a statistic}		
			Sometimes, we only look at our random variable $Y$ through a statistic $S$, such as a test statistic.
			In such situations, it does not matter whether $Y$ is actually invariant; it only matters whether it \emph{looks} invariant when viewed through this statistic.
			This leads us to the following weaker notion of invariance, which recovers the standard notion if $S$ is invertible.
			This will yield a more general condition under which tests for group invariance are valid, given the choice of test statistic.
			
			\begin{dfn}[Invariance through a statistic]\label{dfn:invariance_through_statistic}
				We say that $Y$ \emph{looks} $\mathcal{G}$ invariant through $S$ if $S(\overline{G}Y) \mid O_Y \overset{d}{=} S(Y) \mid O_Y$ almost surely.
			\end{dfn}
			
			In Appendix \ref{appn:invariance_through}, we cover a simple example that illustrates the difference between invariance and invariance through a statistic.

			\begin{rmk}\label{rmk:counterexample}
				It may be tempting to remove the `conditional on the orbit'-component from Definition \ref{dfn:invariance_through_statistic}, and simply demand that $S(\overline{G}Y) \overset{d}{=} S(Y)$.
				A related condition is considered by \citet{kashlak2022asymptotic}.
				While potentially interesting in its own right, this condition is insufficient for testing invariance: we include a counterexample in Appendix \ref{appn:counterexample}.
				There, this condition is satisfied but the random variable is not invariant through $S$, and we find that the resulting classical group invariance test is not valid.
				This counterexample arose in personal communication with Adam Kashlak.
			\end{rmk}

		\subsection{Reducing to a single orbit and representative inversion kernels}\label{sec:inversion_kernels}
			When testing invariance under a group of permutations (exchangeability), it is common to convert data to its ranks (relative to a canonical ordering).
			Similarly, it is common to only look at the signs of random variables that are invariant under sign-flipping (symmetric about zero) and a normalized statistic $y \mapsto y / \|y\|_2$ for rotation-invariant random variables.
			Functions of these statistics give rise to rank tests, sign tests and the $t$-test (see Example \ref{exm:t-test}).
			A desirable property of these statistics is that their null distribution does not depend on the orbit.
			
			Ranks, signs and normalization are special cases of a \emph{representative inversion} $\gamma$ (see \cite{kallenberg2011invariant}, Chapter 7 of \cite{kallenberg2017random} and \cite{chiu2023non}).
			A representative inversion $\gamma$ can be viewed as mapping the data $Y$ back to the group $\mathcal{G}$, such that if $Y$ is $\mathcal{G}$-invariant then $\gamma(Y) \sim \textnormal{Unif}(\mathcal{G})$.
			
 			To define an inversion kernel, it is convenient to first assume that the group $\mathcal{G}$ acts \emph{freely} on $\mathcal{Y}$.
			This means that $Gy = y$ for some $y \in \mathcal{Y}$ implies $G = I$.
			In the context of permutations, this assumption means that there are no ties: indeed, barring ties, a non-identity permutation of data always modifies the original data.
			Under this assumption, we can uniquely define the representative inversion as a map $\gamma : \mathcal{Y} \to \mathcal{G}$ that takes $y$ and returns the element $G$ that carries the representative element $[y]$ on the orbit of $y$ to $y$.
			That is, $\gamma(y)[y] = y$.
							
			If the group action is not free, then there may exist multiple elements in $\mathcal{G}$ that carry $[y]$ to $y$, so that $\gamma(y)$ is not uniquely defined.
			For the non-free setting, we overload the notation of $\gamma$ so that $\gamma(y)$ is uniformly drawn from the elements in $\mathcal{G}$ that carry $[y]$ to $y$, which is well-defined as shown in Theorem 7.14 of \cite{kallenberg2017random}.
			This gives us $\gamma(y) [y] = y$ almost surely.
			Section \ref{sec:exm_permutation} in the Supplementary Material contains a concrete illustration.
			
			\begin{rmk}[Inversion kernels, ranks, signs and normalization]\label{rmk:relationship_rank_inversion_kernel}
				Ignoring ties, the relationship between the inversion kernel and ranks is that ranks are in bijective correspondence to the group of permutations.
				In case of ties, the inversion kernel can be viewed as a slight generalization, that breaks ties through randomization by smearing out the probability mass over different permutations that yield the same data due to ties.
					
				Similarly, barring zeros, the signs of a tuple of data $(X_1, ..., X_n)$ are in bijective correspondence to a group of sign-flips $\{-1, 1\}^n$.
				In $\mathbb{R}^2 \setminus \{0\}$, the normalized vector $y \mapsto y / \|y\|_2$ is in bijective correspondence to the special orthogonal group of `rotations'.
				In higher dimensions, there may be multiple rotations that carry the representative element $[y/ \|y\|_2]$ to $y/ \|y\|_2$, and the resulting inversion kernel can be interpreted as uniformly sampling one of these rotations.
			\end{rmk}
						
	\section{Tests and e-values for group invariance}\label{sec:ph_gi}
		\subsection{Hypothesis and e-value}
			Our goal is to measure evidence against the hypothesis that a random variable $Y$ is drawn from some $\mathcal{G}$ invariant distribution:
			\begin{align*}
				Y \text{ is } \mathcal{G} \text{ invariant}.
			\end{align*}
			Equivalently, we test whether the latent distribution from which $Y$ is sampled is in the collection $H := \{\mathbb{P} : \mathbb{P} \textnormal{ is } \mathcal{G} \textnormal{ invariant}\}$.
			For this purpose, we use an e-value $\varepsilon : \mathcal{Y} \to [0, \infty]$, which is said to be valid for the hypothesis $H$ if 
			\begin{align*}
				\sup_{\mathbb{P} \in H} \mathbb{E}^{\mathbb{P}}\,\varepsilon \leq 1.
			\end{align*}
			We say an e-value is exact for $H$ if $\mathbb{E}^{\mathbb{P}}\,\varepsilon = 1$ for every $\mathbb{P} \in H$.
						
			\begin{rmk}[Exact e-value]
				The term `exact e-value' with respect to a hypothesis $H$ is typically reserved for the weaker property $\sup_{\mathbb{P} \in H}\mathbb{E}^{\mathbb{P}}\,\varepsilon = 1$.
				Our property may be viewed as `uniformly exact', but for brevity we simply refer to it as exact.
			\end{rmk}
			
		\subsection{Characterizing e-values for group invariance}\label{sec:characterization_e-values}	
			We immediately present our first result, which characterizes valid and exact e-values for group invariance.
			It states that an e-value is valid for group invariance if and only if it is valid for a uniform distribution on each orbit.
			A formal proof is provided in Appendix \ref{proof:thm:e-values}.
			
			\begin{thm}[Characterizing e-values]\label{thm:e-values}
				Let $\varepsilon : \mathcal{Y} \to [0, \infty]$. Then,
				\begin{itemize}
					\item[(i)] $\varepsilon$ is a valid e-value for $\mathcal{G}$ invariance if and only if  $\mathbb{E}^{\textnormal{Unif}(O)}\, \varepsilon \leq 1$, for every $O \in \mathcal{Y} / \mathcal{G}$,
					\item[(ii)] $\varepsilon$ is an exact e-value for $\mathcal{G}$ invariance  if and only if $\mathbb{E}^{\textnormal{Unif}(O)}\,\varepsilon = 1$, for every $O \in \mathcal{Y} / \mathcal{G}$.
				\end{itemize}
			\end{thm}
			Throughout, we use that for $y \in O$, $\mathbb{E}^{\textnormal{Unif}(O)}[\varepsilon] =\mathbb{E}_{\overline{G}}[\varepsilon(\overline{G}y)]$, since $\overline{G}y \sim \textnormal{Unif}(O)$.
			
			Because the orbits partition the sample space, our data $Y$ lands in exactly one orbit: $O_Y$.
			As a consequence, we actually only need our e-value to be valid under $\textnormal{Unif}(O_Y)$, as captured in Corollary \ref{cor:conditional_validity}.
			This is the key that facilitates testing group invariance.
			In particular, we may view the problem of testing group invariance as first observing the orbit $O_Y$ and then testing the simple hypothesis that $Y$ is uniform on $\textnormal{Unif}(O_Y)$.

			\begin{cor}[Characterizing through conditioning]\label{cor:conditional_validity}
				Let $\varepsilon : \mathcal{Y} \to [0, \infty]$. Then,
				\begin{itemize}
					\item[(i)] $\varepsilon$ is a valid e-value for $\mathcal{G}$ invariance if and only if  $\mathbb{E}^{\textnormal{Unif}(O_Y)}[\varepsilon]\leq 1$, $O_Y$-a.s.,
					\item[(ii)] $\varepsilon$ is an exact e-value for $\mathcal{G}$ invariance if and only if  $\mathbb{E}^{\textnormal{Unif}(O_Y)}[\varepsilon] = 1$, $O_Y$-a.s.
				\end{itemize}
			\end{cor}
			
			\begin{rmk}[Non-compact groups]\label{rmk:ergodic_non_sequential}
				For a non-compact group $\mathcal{G}$ acting on a Borel space, invariant probabilities still admit a decomposition: any $\mathcal{G}$ invariant probability may be viewed as a mixture over ergodic probabilities (see \citet{kallenberg2021foundations} Theorem 25.24).
				In an analogue of Theorem \ref{thm:e-values}, ergodic probabilities then take the place of the uniform $\textnormal{Unif}(O)$.
				
				Corollary \ref{cor:conditional_validity} falls apart: the conditional probability given $O_Y$ is not the same for each invariant probability.
				Hence, the relevant ergodic probability cannot be identified from the orbit of a single observation.
				This property is key for many of our results, so we assume compactness throughout, only briefly returning to non-compact groups  in Remark \ref{rmk:ergodic}.
			\end{rmk}

		\subsection{Generic traditional tests for group invariance}\label{sec:trad_gi_test}
			Given any test statistic $T : \mathcal{Y} \to \mathbb{R}$, ideally designed to be large under the alternative, the classical group invariance test $\varepsilon_\alpha : \mathcal{Y} \to [0, 1/\alpha]$ is given by
			\begin{align}\label{ineq:group_invariance_test}
				\varepsilon_\alpha(y)
					= \tfrac{1}{\alpha}\mathbb{I}\left\{T(y) > q_\alpha^{\overline{G}}[T(\overline{G}y)]\right\} + \tfrac{c([y])}{\alpha} \mathbb{I}\left\{T(y) = q_{\alpha}^{\overline{G}}\left[T(\overline{G}y)\right]\right\},
			\end{align}
			where $q_\alpha^{\overline{G}}[T(\overline{G}y)]$ denotes the $\alpha$ upper-quantile of the distribution of $T(\overline{G}y)$ for $\overline{G} \sim \textnormal{Unif}(\mathcal{G})$ and $y$ fixed, and $c([y])$ is some orbit-dependent constant.\footnote{$c([y])$ solves $c([y]) \times \mathbb{P}_{\overline{G}_2}\left(T(\overline{G}_2[y]) = q_\alpha^{\overline{G}}[T(\overline{G}[y])]\right) = \alpha - \mathbb{P}_{\overline{G}_2}\left(T(\overline{G}_2[y]) > q_{\alpha}^{\overline{G}}\left[T(\overline{G}[y])\right]\right)$, where $\overline{G}, \overline{G}_2 \sim \textnormal{Unif}(\mathcal{G})$, independently.}
			
			If we ignore the final term in \eqref{ineq:group_invariance_test}, which vanishes in continuous-data settings, this test rejects at level $\alpha$ if the test statistic exceeds an orbit-dependent critical value $q_\alpha^{\overline{G}}[T(\overline{G}y)]$.
			If we do include the final term, then the value $\alpha\varepsilon_\alpha(y)$ is classically interpreted as a probability with which we should subsequently reject the hypothesis using external randomization.
	
			It is well-known that, for any $T$, $\varepsilon_\alpha$ is exact for $\mathcal{G}$-invariance.
			We extend this result, by showing $\varepsilon_\alpha$ is exact for $\mathcal{G}$ invariance through $T$, as in Definition \ref{dfn:invariance_through_statistic}.
			To the best of our knowledge, this aspect is novel.
			The conditional validity implies that the test statistic $T$ may depend on the orbit.
			Its proof can be found in Section \ref{proof:theorem_traditional} in the Supplementary Material.
			
			\begin{thm}\label{thm:gi_p_value}
				If $Y$ looks $\mathcal{G}$ invariant through $T$, then $\mathbb{E}_{Y}[\varepsilon_\alpha \mid O_Y] = 1$, $O_Y$-a.s., which implies $\mathbb{E}_Y[\varepsilon_\alpha] = 1$.
			\end{thm}
			
			\begin{exm}[$t$-test]\label{exm:t-test}
				Suppose $\mathcal{Y} = \mathbb{R}^n$ and $T(y) = \iota'y/\|y\|_2$, where $\iota$ is some unit vector, typically $\iota = n^{-1/2}(1, 1, \dots, 1)$.
				If $Y$ is spherically invariant through $T$, then $T(Y)$ is \textnormal{Beta}$(\tfrac{n-1}{2}, \tfrac{n-1}{2})$-distributed on $[-1, 1]$ (see e.g. \cite{koning2023more} for a proof).
				Equivalently, $\sqrt{n-1}T(Y) / \sqrt{1 - T(Y)^2}$ is $t$-distributed.
				The resulting test for spherical invariance is also known as the $t$-test.
			\end{exm}
			
			\begin{exm}[Conformal prediction]\label{exm:conformal}
				Here, we consider the most basic form of conformal prediction \citep{shafer2008tutorial, angelopoulos2024theoretical}.
				Suppose $\mathcal{Y} = \mathbb{R}^{n+1}$ and $\mathcal{G}$ is the group of permutations acting on the canonical basis of $\mathbb{R}^{n+1}$.
				Let $Y^{n+1}$ be a $\mathcal{G}$ invariant (exchangeable) random variable on $\mathcal{Y}$, and let $T : \mathcal{Y} \to \mathbb{R}$ be a test statistic that only depends on the final element $Y_{n+1}$.
				Suppose we only observe $Y^n = (Y_1, \dots, Y_n)$ and want to test whether the unobserved $Y_{n+1}$ equals $y^*$.
				We can then use the permutation test based on $T((Y^n, y^*))$, which is also known as conformal inference.
				Repeating this test for all $y^* \in \mathcal{Y}$ and collecting the values of $y^*$ for which we do not reject yields the conformal prediction set, which is a confidence set for $Y_{n+1}$ in $\mathbb{R}$.
			\end{exm}
			
		\subsection{Generic e-values for group invariance}\label{sec:e-value_generic}
			We now move beyond traditional group invariance tests, by deriving `generic' exact e-values for group invariance based on some test statistic $T$.
			As with the traditional group invariance tests in Section \ref{sec:trad_gi_test}, we retain great freedom in our selection of the test statistic.
			In particular, let $T$ be some arbitrary non-negative test statistic that satisfies $0 < \mathbb{E}_{\overline{G}}T(\overline{G}y) < \infty$ for every $y \in \mathcal{Y}$.
			
			Based on this test statistic $T$, we consider the e-value
			\begin{align}\label{eq:exact_e_form}
				\varepsilon_T(y) 
					= \frac{T(y)}{\mathbb{E}_{\overline{G}}T(\overline{G}y)},  \quad \textnormal{with }\overline{G} \sim \textnormal{Unif}(\mathcal{G}).
			\end{align}
			
			Theorem \ref{thm:essentially_complete} shows that this e-value is exact, and that any exact e-value for $\mathcal{G}$ invariance may be construed in this manner.
			The proof uses Theorem \ref{thm:e-values}, and is found in Appendix \ref{proof:essentially_complete}.
			
			\begin{thm}\label{thm:essentially_complete}
				An e-value $\varepsilon$ is exact for $\mathcal{G}$ invariance if and only if it is of the form $\varepsilon_T$ for some statistic $T$.
			\end{thm}
			
			By Theorem \ref{thm:essentially_complete}, we can use any appropriately integrable test statistic $T$ to construct an exact e-value for $\mathcal{G}$ invariance.
			In fact, as a non-exact e-value is such a statistic, we can plug it in for $T$ to transform it into an exact e-value.
			We exploit this trick in Appendix \ref{sec:symmetry}.

			Proposition \ref{prp:e-value_invariance_through} shows that we may also relax the assumption of $\mathcal{G}$ invariance by incorporating the choice of the test statistic.
			This assumption is weaker than $\mathcal{G}$ invariance through $T$, which we assume for Proposition \ref{thm:gi_p_value}: we only require the expectation of $T(Y)$ on each orbit to equal the uniform orbit-average $\mathbb{E}_{\overline{G}}T(\overline{G}y)$.
			Its proof is given in Appendix \ref{proof:e-value_invariance_through}.
			
			\begin{prp}\label{prp:e-value_invariance_through}
				Assume that $\mathbb{E}_Y[T(Y) \mid O_Y] = \mathbb{E}_{\overline{G}}[T(\overline{G}Y) \mid O_Y]$ almost surely. 
				Then $\mathbb{E}_Y[\varepsilon_T(Y) \mid O_Y] = 1$, and so $\mathbb{E}_Y[\varepsilon_T(Y)] = 1$.
			\end{prp}

			\begin{exm}[An e-value version of the $t$-test]\label{exm:e-value_t-test}
				Continuing from Example \ref{exm:t-test}, one may desire to derive ``the e-value version'' of the $t$-test.
				But because e-values are a (rich) generalization of binary tests, there is no unique e-value version of the $t$-test: any e-value $\varepsilon_T$ based on a statistic $T$ that is non-decreasing in $\iota'y / \|y\|_2$ could reasonably qualify.
			\end{exm}
			
			\begin{exm}[Conformal inference with e-values]
				Continuing the setup from Example \ref{exm:conformal}, if $T$ is a non-negative test statistic that only depends on the final element, then $T((Y^n, y^*))/\mathbb{E}_{\overline{G}}  T(\overline{G}(Y^n, y^*))$ is an exact e-value for conformal inference.
			\end{exm}

		\subsection{Obtaining the normalization constant and Monte Carlo group invariance e-values}
			The main computational challenge with generic e-values for group invariance is the normalization constant $\mathbb{E}_{\overline{G}} T(\overline{G}Y)$: as the group $\mathcal{G}$ is often large, averaging $T(GY)$ over all $G$ may not be computationally feasible.
			However, the normalization constant can be estimated.
			
			One simple idea is to use a Monte Carlo approach by replacing $\overline{G}$ with a random variable $\overline{G}^M$ that is uniformly distributed on a collection $(\overline{G}^{(1)}, \overline{G}^{(2)}, \dots, \overline{G}^{(M)})$ of $M \geq 1$ mutually independent and identically distributed copies of $\overline{G}$, independent from $Y$.
			Writing $\overline{G}^{(0)} = I$, this yields the `Monte Carlo group invariance e-value'
			\begin{align*}
				\varepsilon_T^M(y) = \frac{T(y)}{\frac{1}{M + 1} \sum_{i = 0}^M T(\overline{G}^{(i)} y)},
			\end{align*}
			and $\varepsilon_T^M(y) = 1$ if the denominator is zero, which can only happen if $T(y) = 0$.
			
			This  e-value is exact in expectation over the Monte Carlo sample, as captured in Theorem \ref{thm:monte_carlo}.
			The proof can be found in Appendix \ref{appn:proof_monte_carlo}, and relies on establishing the exchangeability of $T(\overline{G}^{(0)} Y), \dots, T(\overline{G}^{(M)} Y)$ under the null hypothesis, and then applying Theorem \ref{thm:essentially_complete}.\footnote{We thank an anonymous referee for suggesting this proof strategy.}
			
			\begin{thm}\label{thm:monte_carlo}
				The Monte Carlo e-value $\varepsilon_T^M$ is exact in expectation over the Monte Carlo samples: $\mathbb{E}_{\overline{G}^{(1)}, \dots, \overline{G}^{(M)}} \mathbb{E}^{\textnormal{Unif}(O)}[\varepsilon_T^M] = 1$, for every orbit $O \in \mathcal{Y} / \mathcal{G}$.
			\end{thm}
			
			While the resulting e-value is exact in expectation for any number of Monte Carlo draws $M$, a larger number of draws should generally improve the estimation of the normalization constant $\mathbb{E}_{\overline{G}} T(\overline{G}Y)$, and thereby reduce the sensitivity of $\varepsilon_T^M$ to the drawn sample.
			
			\begin{rmk}[Sequential Monte Carlo e-values]
				\citet{fischer2024sequential} study sequential sampling in the context of traditional Monte Carlo group invariance tests, where the number of draws $M$ is  a stopping time.
				While the underlying idea should generalize, the approach of \citet{fischer2024sequential} is tailored towards making a binary decision, and we believe it would require substantial modification.
				Indeed, \citet{stoepker2024inference} show, in the context of p-values, that a much larger number of Monte Carlo draws is desirable for a continuous measure of evidence when compared to a binary decision.
			\end{rmk}

			\begin{rmk}[Subgroups]
				One may also replace $\overline{G}$ in \eqref{eq:exact_e_form} with a random variable that is uniform on a compact subgroup \citep{chung1958randomization}.
				As invariance under $\mathcal{G}$ implies invariance under all its subgroups, this guarantees the resulting e-value is valid.
				\cite{koning2023more} note that we can strategically select the subgroup based on the test statistic and alternative to increase power.
				\cite{koning2023power} observes that this can even yield methods that are more powerful than using the entire group $\mathcal{G}$.
				Ideas to go beyond uniform distributions on subgroups appear in \citet{hemerik2018exact} and \citet{ramdas2023permutation}.
			\end{rmk}

	\section{Optimal e-values for group invariance}\label{sec:optimal}
		While the approach in Section \ref{sec:ph_gi} is flexible, it does not instruct us how to construct a \emph{good} e-value, which is the topic of this section.

		\subsection{Background: the Neyman-Pearson lemma for e-values}\label{sec:objectives}
			We follow the perspective presented in \citet{koning2024continuous}, which unifies optimal classical testing and optimal e-values.
			The idea is to maximize the expected utility $U$ of the e-value, under an alternative $\mathbb{Q}$.
			
			Maximizing expected utility nests classical power at level $\alpha > 0$ for the `Neyman--Pearson' utility function $U_\alpha(x) = x \wedge 1/\alpha$.
			Indeed, let $\mathbb{P}$ be a simple null hypothesis, and let $\mathbb{Q}_a$ and $\mathbb{Q}_s$ denote the absolutely continuous and singular parts of $\mathbb{Q}$ with respect to $\mathbb{P}$.
			Then, we recover the Neyman-Pearson lemma, since an optimal e-value under $U$ is
			\begin{align}\label{eq:NP}
				\varepsilon^*
					= 
					\begin{cases}
						1/\alpha, &\text{ if } d\mathbb{Q}_a/d\mathbb{P} > c_\alpha, \\
						k, &\text{ if } d\mathbb{Q}_a/d\mathbb{P} = c_\alpha, \\
						0 &\text{ if } d\mathbb{Q}_a/d\mathbb{P}< c_\alpha,
					\end{cases}
					\quad \mathbb{P}\textnormal{-a.s.},
			\end{align}
			for suitable constants $c_\alpha \geq 0$ and $k \in [0, 1/\alpha]$, choosing $\varepsilon^* = 1/\alpha$ (or $\varepsilon^* = \infty$), $\mathbb{Q}_s$-almost surely.
			Here, $\varepsilon^* = 1/\alpha$ corresponds to a rejection at level $\alpha$ and $\varepsilon^* = 0$ to a non-rejection.
			
			The sole innovation of the e-value perspective here is to directly interpret $\varepsilon^* = k$ as evidence against the null, instead of the classical interpretation as an instruction to reject with probability $\alpha\varepsilon^* \in [0,1]$.
			Since $\varepsilon^* = k$ typically happens with low or zero probability, the e-value has little to offer if we stick to the Neyman-Pearson utility.
			For this reason, we \emph{must} let go of the traditional notion of power if we wish to truly use the potential of e-values.

			Lemma \ref{lem:U-optimal}, paraphrased from \citet{koning2024continuous}, may be viewed as a `Neyman--Pearson lemma' for e-values.
			There, we consider a more general non-decreasing utility function $U : [0, \infty] \to [-\infty, \infty]$ with $U(\infty) = \limsup_{x \to \infty} U(x)$.
			We say that $\varepsilon^*$ is $U$-optimal if 
			\begin{align*}
				\mathbb{E}^{\mathbb{Q}_a}[U(\varepsilon^*)] \geq \mathbb{E}^{\mathbb{Q}_a}[U(\varepsilon)],
			\end{align*}
			for all $\varepsilon \in \mathcal{E}(\mathbb{P}) := \{\varepsilon : \textnormal{measurable and }\mathbb{E}^{\mathbb{P}}[\varepsilon] \leq 1\}$, and imposing $\varepsilon = \infty$, $\mathbb{Q}_s$-almost surely, assuming that both sides are well-defined.
			Note that we may restrict to $\mathcal{E}_1(\mathbb{P}) := \{\varepsilon \in \mathcal{E}(\mathbb{P}): \mathbb{E}^{\mathbb{P}}[\varepsilon] = 1\}$, since valid e-values outside this class can be weakly improved.
			For completeness, we include a simple proof in Appendix \ref{proof:U-optimal}.
			
			\begin{lem}[\citet{koning2024continuous}]\label{lem:U-optimal}
				Let $U : [0, \infty] \to [-\infty, \infty]$ be non-decreasing.
				Fix $\lambda \geq 0$ and let
				\begin{align*}
					\varepsilon_\lambda(y) \in \textnormal{argmax}_{x \in [0, \infty]} \frac{d\mathbb{Q}_a}{d\mathbb{P}}(y) U(x) - \lambda x, \quad \mathbb{P}\textnormal{-a.s.},
				\end{align*}
				and $\varepsilon_\lambda = \infty$, $\mathbb{Q}_s$-a.s.
				Suppose that $\varepsilon_{\lambda}\in \mathcal{E}_1(\mathbb{P})$.
				Then $\varepsilon_{\lambda}$ is $U$-optimal.
			\end{lem}
			
			Here, we take $\frac{d\mathbb{Q}_a}{d\mathbb{P}}(y) U(x) - \lambda  x$ at $x = \infty$ as the limsup.
			If $U$ is concave, we can write
			\begin{align*}
				\lambda \Bigg/\frac{d\mathbb{Q}_a}{d\mathbb{P}}(y) \in \partial U(\varepsilon_\lambda(y)), \quad \mathbb{Q}_a\textnormal{-a.s.},
			\end{align*}
			on $\{\frac{d\mathbb{Q}_a}{d\mathbb{P}} > 0\}$. 
			If $U$ is differentiable with invertible derivative $U'$, then $\varepsilon_\lambda = (U')^{-1}(\lambda /\frac{d\mathbb{Q}_a}{d\mathbb{P}})$.
			
			The lemma reduces the problem of finding an optimal e-value to finding the constant $\lambda$.
			To show the existence of such $\lambda$, we can assume $U$ is upper semicontinuous to allow for a measurable selection of $\varepsilon_\lambda$, and assume sub-log growth of the utility: $U(y) - U(x) \leq K \log(y/x)$ for $0 < x < y < \infty$.
			A weak condition for the objective to be well-defined at $\varepsilon_\lambda$ is $U(x) > -\infty$ for some $x \in [0, \infty)$.
			
		\subsection{Characterizing e-values and optimal e-values for group invariance}
			While optimal e-values for simple null hypotheses are well-understood, optimal e-values for composite hypotheses are significantly more challenging to characterize.
			In Theorem~\ref{thm:characterize_optimal}, we present the key result for deriving optimal e-values for $\mathcal{G}$ invariance.
			It shows that if an e-value is `locally' valid and optimal on every orbit, then it is valid for $\mathcal{G}$ invariance and `globally' optimal.
			Its proof is found in Appendix \ref{proof:characterize_optimal}.
			
			The result assumes that we are maximizing an objective $K$ that can be decomposed into local objectives $K_O$ on each orbit.
			This condition seems quite mild, and holds for the expected utility-type objectives discussed in Section \ref{sec:objectives}.
			
			In order to present the result, let $F_+^\mathcal{A}$ denote the set of $[0, \infty]$-valued measurable functions on the space $\mathcal{A}$.
			For a measurable subspace $\mathcal{B} \subseteq \mathcal{A}$, we use $f_{|\mathcal{B}}$ to denote the restriction of $f \in F_+^{\mathcal{A}}$ to $\mathcal{B}$.
			Note that $f_{|\mathcal{B}} \in F_+^{\mathcal{B}}$.
			We use  $K : F_+^{\mathcal{Y}} \to [0, \infty]$ to denote the aggregate objective, which we define by
			\begin{align}\label{eq:Psi_aggr}
				K(f) 
					= \Psi\left(\bigl(K_O\bigl(f_{|O}\bigr)\bigr)_{O \in \mathcal{Y} / \mathcal{G}}\right),
			\end{align}
			where $K_O : F_+^O \to [0, \infty]$ is an orbit-level objective, and the aggregating function $\Psi : [0, \infty]^{\mathcal{Y} / \mathcal{G}} \to [0, \infty]$ is non-decreasing in each of its inputs.
				
			\begin{thm}[Local optimality $\implies$ global optimality]\label{thm:characterize_optimal}
				Let $\varepsilon^* \in F_+^\mathcal{Y}$. Suppose that for each $O \in \mathcal{Y} / \mathcal{G}$:
				\begin{itemize}
					\item[(i)] $\varepsilon_{|O}^*$ is a valid e-value for $\mathrm{Unif}(O)$,
					\item[(ii)] $K_O\left(\varepsilon_{|O}^*\right) \geq K_O(\varepsilon)$ for \emph{every} e-value $\varepsilon \in F_+^O$ that is valid for $\mathrm{Unif}(O)$.
				\end{itemize}
				Then $\varepsilon^*$ is valid for $\mathcal{G}$ invariance and $K$-optimal: $K(\varepsilon^*) \geq K(\varepsilon)$ for every e-value $\varepsilon$ that is valid for $\mathcal{G}$ invariance.
			\end{thm}
			
			\begin{rmk}[Optimal uniformly in aggregation functions]\label{rmk:uniform_in_aggregators}
				An implication of Theorem \ref{thm:characterize_optimal} is that an e-value that is optimal for every $K_O$, $O \in \mathcal{Y} / \mathcal{G}$, is also optimal for \emph{any} choice of aggregation function $\Psi$.
				We exploit this idea in Section \ref{sec:optimal_orbit}.
			\end{rmk}
			
			\begin{rmk}[Beyond group invariance]
				The result goes through if the partitioning of the sample space is not generated by a group, and the null hypothesis is simply some probability mixture over known distributions on each of the subsets in the partition.
				Our group structure naturally generates such a partition, where the distribution on each subset in the partition is uniform.
				This approach may be interesting for deriving optimal e-values in other contexts.
			\end{rmk}

		\subsection{Expected utility-optimal e-values for group invariance: alternative on sample space}\label{sec:optimal_sample_space}
			We now apply Theorem~\ref{thm:characterize_optimal} to derive expected-utility optimal e-values for group invariance.
			A remarkable feature of these optimal e-values is that they may be expressed in terms of $\mathbb{Q}$ and its symmetrization $\overline{\mathbb{Q}}$.
			This is a consequence of Lemma \ref{lem:conditional_unconditional_RN}, which shows that the likelihood ratio $d\mathbb{Q} / d\overline{\mathbb{Q}}$ coincides with the conditional likelihood ratio.

			In this section, let $\mathbb{Q} = \mathbb{Q}_a + \mathbb{Q}_s$ be the Lebesgue decomposition of $\mathbb{Q}$ with respect to $\overline{\mathbb{Q}}$.
			Moreover, let $\pi(y) = O_y$ and define the shared probability $\mu = \mathbb{Q} \circ \pi^{-1} = \overline{\mathbb{Q}} \circ \pi^{-1}$ on $\mathcal{Y} / \mathcal{G}$.
	
			\begin{lem}\label{lem:conditional_unconditional_RN}
				The function $y \mapsto \frac{d\mathbb{Q}_a(\cdot \mid O_y)}{d\textnormal{Unif}(O_y)}(y)$ is a version of $\frac{d\mathbb{Q}_a}{d\overline{\mathbb{Q}}}$, for some subprobability kernel $\mathbb{Q}_a(\cdot \mid O)$ satisfying $\mathbb{Q}_a(\cdot) = \int \mathbb{Q}_a(\cdot \mid O) d\mu(O)$ and $\mathbb{Q}_a(\cdot \mid O) \ll \textnormal{Unif}(O)$, $\mu$-a.s.
			\end{lem}			
			
			Combining this lemma with Theorem \ref{thm:characterize_optimal} and Lemma \ref{lem:U-optimal} yields Theorem \ref{thm:utility-optimal-invariant}, which characterizes expected-utility-optimal e-values for group invariance.
			Its proof is in Appendix \ref{proof:utility-optimal-invariant}.
			
			\begin{thm}[$U$-optimal $\mathcal G$-invariant e-value]\label{thm:utility-optimal-invariant}
				Assume $U : [0, \infty] \to [-\infty, \infty]$ is non-decreasing.
				Let $\varepsilon^U$ be such that for every orbit $O$:
				\begin{itemize}
					\item[(i)]  $\varepsilon^U$ is an exact e-value for $\textnormal{Unif}(O)$,
					\item[(ii)] 	$\varepsilon^U(y) \in \textnormal{argmax}_{x \in [0, \infty]} \frac{d\mathbb{Q}_a}{d\overline{\mathbb{Q}}}(y) U(x) - \lambda_O x$, $\textnormal{Unif}(O)$-a.s., for some $\lambda_O \geq 0$,
				\end{itemize}
				and $\varepsilon^U = \infty$, $\mathbb{Q}_s$-a.s.
				Then $\varepsilon^U$ is a exact e-value for $\mathcal{G}$ invariance and $U$-optimal against $\mathbb{Q}$.
			\end{thm}
			
			We now apply this result to derive several corollaries for different kinds of utility functions, starting with the Neyman--Pearson utility function $U_\alpha(x) = x \wedge 1/\alpha$, for some fixed $\alpha \in (0, 1]$.
			This recovers the main result of \citet{lehmann1949theory}.
			One difference is that they seem to implicitly assume the existence of a $\mathcal{G}$ invariant reference measure.
			Moreover, we permit $\varepsilon = \infty$ on $\mathbb{Q}_s$, but this may be replaced by $\varepsilon = 1/\alpha$.
			This e-value corresponds to the classical group invariance test with the test statistic $T(y) = \frac{d\mathbb{Q}_a}{d\overline{\mathbb{Q}}}(y)$.

			\begin{cor}[Neyman--Pearson]\label{cor:lehman-stein}
				Consider the utility $U(x) = x \wedge 1/\alpha$, $\alpha \in (0, 1]$.
				Suppose that for every orbit $O$, there exist constants $\lambda_O \geq 0$ and $k_O \in [0, 1/\alpha]$ such that 
				\begin{align*}
					\varepsilon^{\textnormal{NP}}
						=
						\begin{cases}
							0, &\textnormal{ if }\frac{d\mathbb{Q}_a}{d\overline{\mathbb{Q}}} < \lambda_O, \\
							k_O, &\textnormal{ if }\frac{d\mathbb{Q}_a}{d\overline{\mathbb{Q}}} = \lambda_O, \\
							1/\alpha, &\textnormal{ if }\frac{d\mathbb{Q}_a}{d\overline{\mathbb{Q}}} > \lambda_O, \\
						\end{cases}
				\end{align*}
				$\textnormal{Unif}(O)$-a.s. and $\varepsilon^{\textnormal{NP}}$ is an exact e-value for $\textnormal{Unif}(O)$, and $\varepsilon^{\textnormal{NP}} = \infty$, $\mathbb{Q}_s$-a.s.
				Then $\varepsilon^{\textnormal{NP}}$ is Neyman--Pearson optimal.
			\end{cor}
			
			Next, we consider the popular log-utility, for which we can easily explicitly characterize the normalization constant.
			If  $\mathbb{Q} \ll \overline{\mathbb{Q}}$ then the normalization constant equals 1 so that $\varepsilon^\textnormal{log} = d\mathbb{Q} / d\overline{\mathbb{Q}}$.
			This happens, for example, if the group is finite as this means $\overline{\mathbb{Q}}$ is a finite average of measures, one of which is $\mathbb{Q}$.
			\begin{cor}[Log-optimal]\label{cor:log-optimal}
				Consider the utility $U = \log$.
				A $U$-optimal e-value is
				\begin{align*}
					\varepsilon^{\textnormal{log}}(y)
						= \left. \frac{d\mathbb{Q}_a}{d\overline{\mathbb{Q}}}(y) \middle/ \mathbb{E}_{\overline{G}}\left[\frac{d\mathbb{Q}_a}{d\overline{\mathbb{Q}}}(\overline{G}y)\right]\right.,
				\end{align*}	
				on orbits with $\mathbb{E}_{\overline{G}}\left[\frac{d\mathbb{Q}_a}{d\overline{\mathbb{Q}}}(\overline{G}y)\right] > 0$, $\varepsilon^{\textnormal{log}} = 1$ on other orbits, and $\varepsilon^{\textnormal{log}}(y) = \infty$, $\mathbb{Q}_s$-a.s.
			\end{cor}

			Log-utility inherits the simple characterization of the normalization constant from the generalized-means, or `power utility' $U_h(x) = (x^h - 1) / h$ for $h \neq 0$ and $U_0(x) = \log(x)$, $h < 1$.
			This provides a parameter $h$ which tunes the `riskiness' of the e-value, where $h \to 1$ yields an all-or-nothing-style e-value, whereas $h \to -\infty$ yields the constant e-value.
			
			\begin{cor}[Generalized-mean-optimal]\label{cor:h-optimal}
				Assume $\mathbb{E}_{\overline{G}}\left[\left\{d\mathbb{Q}_a/d\overline{\mathbb{Q}}(\overline{G}y)\right\}^{\frac{1}{1 - h}}\right] < \infty$.
				Then, a $U_h$-optimal e-value is given by
				\begin{align*}
					\varepsilon^{(h)}(y)
						= \left. \left(\frac{d\mathbb{Q}_a}{d\overline{\mathbb{Q}}}(y)\right)^{\frac{1}{1 - h}} \middle/ \mathbb{E}_{\overline{G}}\left[\left(\frac{d\mathbb{Q}_a}{d\overline{\mathbb{Q}}}(\overline{G}y)\right)^{\frac{1}{1 - h}}\right]\right.,
				\end{align*}
				on orbits with $\mathbb{E}_{\overline{G}}\left[\left\{d\mathbb{Q}_a/d\overline{\mathbb{Q}}(\overline{G}y)\right\}^{\frac{1}{1 - h}}\right] > 0$, $\varepsilon^{(h)} = 1$ on other orbits and $\varepsilon^{(h)} = \infty$, $\mathbb{Q}_s$-a.s.
			\end{cor}

			\begin{rmk}[Link to `generic' e-values]
				These corollaries give guidance to the choice of test statistic $T$ for the `generic' e-value presented in Section \ref{sec:e-value_generic}.
				Indeed, $T \propto d\mathbb{Q}_a/d\overline{\mathbb{Q}}$ and $T \propto (d\mathbb{Q}_a/d\overline{\mathbb{Q}})^{\frac{1}{1-h}}$ yield log-optimal and generalized-mean optimal e-values.
			\end{rmk}
			
			In case a $\sigma$-finite $\mathcal{G}$ invariant reference measure $\mathbb{H} \gg \overline{\mathbb{Q}}, \mathbb{Q}$ is available, we can replace $\overline{\mathbb{Q}}$ by $\mathbb{H}$, even if it is not a probability measure.
			This is convenient, because densities are often presented with respect to some dominating invariant reference measure, such as the Lebesgue measure.
			We leverage this in Appendix \ref{sec:gaussian}.
			
			\begin{prp}\label{prp:invariant_reference}
				In Theorem~\ref{thm:utility-optimal-invariant} and its corollaries we may replace $\overline{\mathbb{Q}}$ by $\mathbb{H} \gg \overline{\mathbb{Q}}, \mathbb{Q}$.
			\end{prp}
			\begin{proof}
				As both $\overline{\mathbb{Q}}$ and $\mathbb{H}$ are $\mathcal{G}$-invariant, $\frac{d\overline{\mathbb{Q}}}{d\mathbb{H}}$ is $\mathcal{G}$-invariant, thus constant on each orbit and so absorbed into the multiplier $\lambda_O$.
			\end{proof}

		\subsection{Expected utility-optimal e-values for group invariance: alternatives on orbits}\label{sec:optimal_orbit}
			In Section~\ref{sec:optimal_sample_space}, we specified a single alternative $\mathbb{Q}$ on the entire sample space $\mathcal{Y}$.
			We may also turn Theorem~\ref{thm:characterize_optimal} around by specifying an alternative $\mathbb{Q}^O$ on each orbit,  resulting in an e-value that is optimal \emph{uniformly in every marginal distribution over orbits}, following Remark~\ref{rmk:uniform_in_aggregators}.
			
			We present this in Corollary \ref{cor:uniform_over_marginals}, which follows from observing in Lemma \ref{lem:conditional_unconditional_RN} that $d\mathbb{Q}_a /d\overline{\mathbb{Q}}$ does not depend on the marginal distributions over orbits.
			Here, we define $\mathbb{Q}^\mu := \int \mathbb{Q}^O d\mu(O)$, for a probability measure $\mu$ on $\mathcal{Y}/\mathcal{G}$.

			\begin{cor}[$U$-optimality uniformly over marginals]\label{cor:uniform_over_marginals}
				If we replace $d\mathbb{Q}_a / d\overline{\mathbb{Q}}$ in Theorem \ref{thm:utility-optimal-invariant} by $d\mathbb{Q}_a^O / d\textnormal{Unif}(O)$, then $\varepsilon^U$ is exact and $U$-optimal against $\mathbb{Q}^\mu$, for every $\mu$.
			\end{cor}

			While specifying an alternative $\mathbb{Q}^O$ on each orbit may sound like an arduous exercise, the following example shows its practical relevance.
			We put this example to practice in Section \ref{sec:hot_hand}, where we apply it to the real-world experiment of \citet{gilovich1985hot}.
			
			In practice, we need only formulate the conditional alternative on the observed orbit.
			More precisely, when we receive our data $Y$, we may first classify its orbit $O_Y$, then choose an alternative $\mathbb{Q}^{O_Y}$ based on $O_Y$ (not based on $Y$ itself), and then compute the e-value using $Y$.
			
			\begin{exm}[Hot hand]\label{exm:hot_hand}
				The hot hand comes from basketball.
				It describes a momentum effect, in which a player hitting a shot increases their probability of hitting subsequent shots.
				It was first statistically popularized by \citet{gilovich1985hot}.
				It is often examined by testing whether a sequence of shot outcomes $Y$ (hit/miss) is exchangeable against some sequential-dependence alternative that describes the hot hand effect \citep{miller2018surprised, ritzwoller2022uncertainty}.
				
				Permuting such a shot sequence exactly fixes the number of hits (\#hit) and misses (\#miss), so that the permutation orbits may be labeled by (\#hit, \#miss).
				If we were to follow Section \ref{sec:optimal_sample_space}, we would need to specify a marginal alternative over the orbits and so over (\#hit, \#miss).
				This is hard, as it requires knowledge of the skill of the player, opponents and teammates.
				
				The strategy in Corollary \ref{cor:uniform_over_marginals} relieves us from specifying an alternative over (\#hit, \#miss): we merely need to specify the conditional distribution of the order of the hits and misses under the hot hand, given the statistic (\#hit, \#miss).
				Here, we may even decide to use the number of hits and misses to influence the strength of the hot hand, as a proxy for skill.
			\end{exm}

		\subsection{Optimal e-values for group invariance: objective on group}\label{sec:optimal_on_group}
			A final strategy is to specify an alternative $\mathbb{Q}^{\mathcal{G}}$ on the group $\mathcal{G}$ itself and test this against the Haar measure $\textnormal{Unif}(\mathcal{G})$ on the group.
			Such an e-value is valid for $\mathcal{G}$ invariance if and only if $\mathbb{E}_{\overline{G}}[\varepsilon(\overline{G})] \leq 1$.
			This may sound overly exotic, but this actually underlies popular rank tests and sign tests.
			
			To apply this idea, we may use an inversion kernel, as discussed in Section \ref{sec:inversion_kernels}.
			In particular, we may derive an optimal e-value $\varepsilon_{\mathcal{G}}^* : \mathcal{G} \to [0, \infty]$ for measuring evidence against the Haar measure, and evaluate it on our data by passing it through an inversion kernel $\varepsilon_{\mathcal{G}}^*(\gamma(Y))$.
			 
			A concrete example is given by the rank statistic, which is in bijection with the group of permutations if we ignore ties.
			This means we may reimagine the function $\text{Rank} : \mathbb{R}^M \to \{1, \dots, M\}^M$ as a function that maps an observation on $\mathbb{R}^M$ to a group of permutations $\text{Rank} : \mathbb{R}^M \to \mathcal{G}$.
			Under exchangeability (permutation invariance), the distribution of the resulting rank is uniform (Haar) on the set of possible ranks.
			An alternative on the group corresponds to any other distribution on the ranks.
			If desired, such an alternative on the group may be obtained by pushing forward an alternative on $\mathcal{Y}$ to $\mathcal{G}$ through the inversion kernel.
				
	\section{E-processes for group invariance}\label{sec:e-processes}		
		\subsection{Characterizing e-processes for group invariance}\label{sec:characterizing_e-processes}
			In this section, we describe sequential gathering of evidence against group invariance through e-processes.
			
			In the sequential setting, we equip our sample space $(\mathcal{Y}, \mathcal{I})$ with a filtration $(\mathcal{I}_n)_{n \geq 0}$, $\mathcal{I}_n \subseteq \mathcal{I}$. 
			An e-process for hypothesis $H$ is a stochastic process $(\varepsilon_n)_{n \geq 0}$, with $\varepsilon_n : \mathcal{Y} \to [0, \infty]$, that is adapted to a filtration $(\mathcal{I}_n)_{n \geq 0}$.
			We say an e-process is anytime valid if $\varepsilon_\tau$ is a valid e-value for every stopping time $\tau$ adapted to $(\mathcal{I}_n)_{n \geq 0}$.
			Without loss of generality, we consider $\mathcal{I}_0 = \{\emptyset, \mathcal{Y}\}$, so that we may impose $\varepsilon_0 = 1$.
			
			In Theorem \ref{thm:e-process}, we characterize e-processes for $\mathcal{G}$ invariance.
			This result may be viewed as a sequential analogue of Theorem \ref{thm:e-values}, and follows directly from it.
			
			\begin{thm}[Characterizing e-processes]\label{thm:e-process}
				$(\varepsilon_n)_{n \geq 0}$ is anytime valid for $\mathcal{G}$ invariance if and only if $\varepsilon_\tau$ is valid for $\textnormal{Unif}(O)$, for every orbit $O$ and stopping time $\tau$.
			\end{thm}
			
			The key challenge when constructing e-processes for group invariance is that the sequential analogue of Corollary \ref{cor:conditional_validity} is typically not useful.
			The problem is that the orbit $O_Y$ is typically not measurable along the filtration $(\mathcal{I}_n)_{n \geq 0}$, and only determined by the $\mathcal{I}$-measurable `terminal data' $Y$.
			As a consequence, while Corollary \ref{cor:conditional_validity_e-process} provides a valid characterization, it conditions on information that is unavailable when we need it.
			
			\begin{cor}[Characterizing through conditioning]\label{cor:conditional_validity_e-process}
				$(\varepsilon_n)_{n \geq 0}$ is a valid e-process for $\mathcal{G}$ invariance if and only if $\mathbb{E}^{\textnormal{Unif}(O_Y)}[\varepsilon_\tau] \leq 1$, $O_Y$-a.s.
			\end{cor}

		\subsection{Tracking orbit-wise martingales}	
			As we cannot generally use Corollary \ref{cor:conditional_validity_e-process} to construct an anytime valid e-process, we pass to an equivalent definition: an e-process is anytime valid for a hypothesis $H$ if and only if for every $\mathbb{P} \in H$ it is $\mathbb{P}$-almost surely bounded from above by a $\mathbb{P}$-non-negative martingale starting at 1 \citep{ramdas2022admissible}.
			
			To put this alternative definition into practice, we could track an entire family of such martingales $(\varepsilon_n^{\mathbb{P}})_{\mathbb{P} \in H}$, and take the e-process as some lower bound.
			A lower bound is the greatest measurable lower bound of the family of martingales, assuming it exists (which is not generally guaranteed if $H$ is uncountable and no dominating measure exists).
			
			In case of a large non-parametric hypothesis such as group invariance, this is a large family to track.
			Luckily, Theorem \ref{thm:infimum_of_martingales} shows that we may reduce this to tracking a $\textnormal{Unif}(O)$-martingale $(\varepsilon_n^{O})_{n \geq 0}$ for every orbit $O \in \mathcal{Y} / \mathcal{G}$.
			Its proof is presented in Appendix \ref{proof:infimum_of_martingales}.
			
			\begin{thm}\label{thm:infimum_of_martingales}
				An e-process $(\varepsilon_n)_{n \geq 0}$ is anytime valid if and only if it is bounded by a non-negative $\textnormal{Unif}(O)$-martingale that starts at 1, $\textnormal{Unif}(O)$-a.s., for every orbit $O \in \mathcal{Y} / \mathcal{G}$.
			\end{thm}
			
			\begin{rmk}[Comments on Theorem \ref{thm:infimum_of_martingales}]\label{rmk:remark_on_theorem_infimum_of_martingales}
				If the orbit is known from the start (i.e. there is only one orbit) then we can take our e-process to be a $\textnormal{Unif}(O)$-martingale.
				As being a $\textnormal{Unif}(O)$-martingale does not restrict the behavior of the e-process outside of the orbit $O$, we suggest to impose $\varepsilon_n^O(y) = \infty$ for $y \not\in O$.
				This may be interpreted as `dropping' a martingale from the family as soon as its associated orbit is no longer a candidate to be the eventual orbit of the data, since this means they cannot affect a lower bound.
			\end{rmk}

		\subsection{Constructing orbit-wise martingales out of an e-value for $\mathcal{G}$ invariance}\label{sec:e-process from e-value}
			The decomposition into orbit-wise martingales is in harmony with the orbit-wise optimality theory from Section \ref{sec:optimal} through the Doob martingale strategy of \citet{koning2025sequentializing}.
			
			In particular, we may start by constructing an optimal e-value $\varepsilon$ for group invariance by leveraging Corollary \ref{cor:uniform_over_marginals}: for each orbit $O$, we choose its restriction $\varepsilon_{|O}$ to be valid for $\textnormal{Unif}(O)$ and optimal against $\mathbb{Q}^O$.
			Next, we induce a Doob martingale $(\varepsilon_n^O)_{n \geq 0}$ with $\varepsilon_n^O = \mathbb{E}^{\textnormal{Unif}(O)}[\varepsilon_{|O} \mid \mathcal{I}_n]$, for every orbit $O$.
			An e-process is then valid if it is $\textnormal{Unif}(O)$-a.s. upper bounded by such $\textnormal{Unif}(O)$-martingales.

			Remarkably, this construction builds towards the global e-value $\varepsilon$ in the following sense.
			For every orbit $O$, if $\varepsilon$ is measurable at stopping time $\tau$, then $\varepsilon_{|O}$ is measurable at $\tau$ and so
			\begin{align*}
				\varepsilon_\tau^O
					= \mathbb{E}^{\textnormal{Unif}(O)}[\varepsilon_{|O} \mid \mathcal{I}_\tau]
					= \varepsilon_{|O}, \quad \textnormal{Unif}(O)\textnormal{-a.s}.
			\end{align*}
			As a consequence, if the realized orbit $O_Y$ is $\mathcal{I}_\tau$-measurable then
			\begin{align*}
				\varepsilon_\tau^{O_Y}(Y) = \varepsilon(Y), \quad \textnormal{Unif}(O_Y)\textnormal{-a.s.}.
			\end{align*}
			Example \ref{exm:e-process_toy} illustrates such an e-process.
			
			\begin{exm}[Illustration of e-process]\label{exm:e-process_toy}
				We consider a simple example of a non-trivial e-process, to illustrate the results in this section.
				
				We will sequentially observe a pair of letters, $\mathcal{Y} = \{AB, BA, AC, CA\}$.
				Let $\mathcal{G}$ be the group that permutes the two letters.
				This means we have two orbits: $O_1 = \{AB, BA\}$ and $O_2 = \{AC, CA\}$.
				Let $Y$ denote our $\mathcal{Y}$-valued random variable, $Y^1$ the first letter of $Y$ and $Y^2 = Y$.
				Let $\mathcal{I}_1 = \sigma(Y^1)$ and $\mathcal{I}_2 = \sigma(Y^2) = \sigma(Y)$, so we sequentially observe two letters.
				
				As an example, we consider the log-optimal e-value against any aggregation of the orbit-wise alternatives $\mathbb{Q}_1(AB) = 2/3$, $\mathbb{Q}_1(BA) = 1/3$ and $\mathbb{Q}_2(AC) = 1/3$, $\mathbb{Q}_2(CA) = 2/3$.
				By Corollary \ref{cor:uniform_over_marginals}, this is given by 
				\begin{align*}
					\varepsilon_2(AB) = 4/3, \quad \varepsilon_2(BA) = 2/3, \quad \varepsilon_2(AC) = 2/3, \quad \varepsilon_2(CA) = 4/3,
				\end{align*}
				which automatically gives the value of the e-process at time 2.
				
				To obtain the e-process at time 1, we now restrict these e-values to the orbits and set the restrictions to $\infty$ elsewhere.
				Then, we apply the Doob martingale strategy to obtain		
				\begin{align*}
					\varepsilon_1^{O_1}(A) &= 4/3, & \varepsilon_1^{O_1}(B) &= 2/3, & \varepsilon_1^{O_1}(C) &= \infty, \\
					\varepsilon_1^{O_2}(A) &= 2/3, & \varepsilon_1^{O_2}(B) &= \infty,     & \varepsilon_1^{O_2}(C) &= 4/3.
				\end{align*}
				Minimizing over the two orbits yields the e-process at time 1:
				\begin{align*}
					\varepsilon_1(A) = 2/3, \quad \varepsilon_1(B) = 2/3, \quad \varepsilon_1(C) = 4/3.
				\end{align*}
				We stress that this e-process is \emph{not} a supermartingale for $\mathcal{G}$ invariance, as
				\begin{align*}
					\sup_{\mathbb{P} : \mathcal{G}\textnormal{-invariant}} \mathbb{E}^{\mathbb{P}}[\varepsilon_2 \mid Y^1 = A] 
						&\geq \max_i\mathbb{E}^{\textnormal{Unif}(O_i)}[\varepsilon_2 \mid Y^1 = A] \\
						&= \max_{i \in \{1, 2\}} \varepsilon_1^{O_i}(A)
						= 4/3 > 2/3 = \varepsilon_1(A).
				\end{align*}
			\end{exm}
		
			\begin{rmk}[Non-compact groups and de Finetti]\label{rmk:ergodic}
				As in the non-sequential setting discussed in Remark  \ref{rmk:ergodic_non_sequential}, we may generalize the characterization of e-processes to non-compact groups through an ergodic theorem.
				Here, the uniform probabilities on orbits are replaced by ergodic probabilities.
				To construct an e-process in such a setting, we may then track an infimum of martingales for the ergodic probabilities, instead of tracking an infimum over martingales for each $\mathcal{G}$ invariant probability.
				
				Of course, without imposing additional structure, tracking a martingale for each ergodic probability may still be a daunting task.
				To highlight this, we may consider perhaps the most famous example of an ergodic theorem: de Finetti's theorem.
				Under suitable regularity conditions, it states that if an infinite sequence is exchangeable (invariant under permutations that move finitely many elements), then its law may be written as a mixture over i.i.d. probabilities (the ergodic measures).
				This has been explored by \citet{ramdas2022testing} in a binary and $d$-ary setting, to show the existence of e-processes in settings where no powerful martingales exist.
				Unfortunately, tracking a martingale for each i.i.d. probability is practically difficult beyond simple examples such as binary data.
			\end{rmk}
			
		\section{Test martingales for group invariance} \label{sec:seq_invariance}
			In this section, we consider test martingales for group invariance.			
			A test martingale for a hypothesis $H$ is a non-negative supermartingale $(\varepsilon_n)_{n \geq 0}$ that starts at 1, $\varepsilon_0 = 1$, and satisfies
			\begin{align}\label{ineq:test_martingale}
				\mathbb{E}^{\mathbb{P}}[\varepsilon_{n+1} \mid \mathcal{I}_n] \leq \varepsilon_{n},
			\end{align}
			for every $n \geq 0$ and every $\mathbb{P} \in H$.
			Such a test martingale is also an anytime valid e-process.
			
			There are three ways to look at test martingales.
			The first, and perhaps more common, is as a practical alternative to an e-process, since test martingales are often easier to construct.
			The second is as an e-process that is anytime valid for the fork-convex hull of a hypothesis $H$ \citep{ramdas2022testing}.
			The third way is as an e-process that is not just anytime valid, but \emph{adaptively anytime valid}: for every pair of stopping times $\sigma < \tau$,
			\begin{align*}
				\mathbb{E}^{\mathbb{P}}[\varepsilon_{\tau} \mid \mathcal{I}_\sigma] \leq \varepsilon_{\sigma}, \textnormal{ for every } \mathbb{P} \in H.
			\end{align*}
			Standard anytime validity may be viewed as only requiring this for $\sigma = 0$.
			As test martingales satisfy a stronger validity condition, this necessarily makes them less powerful than e-processes in settings where this additional guarantee is not important.
			
			Unlike in the e-process setting, we do not allow arbitrary filtrations here, but we consider a natural filtration that matches the group setting, sequentially revealing a growing sequence of subgroups $(\mathcal{G}_n)_{n \geq 0}$ and data $(X^n)_{n \geq 0}$.
			At each point in time $n$, we show how to construct an e-value for $\mathcal{G}_n$ invariance of $X^n$ conditional on the past data $X^{n-1}$, and then construct a test martingale as their sequential product.
			
			In Appendix \ref{sec:impoverishing}, we reduce to a single orbit.
			This yields a simple null hypothesis: a uniform distribution on this orbit, so that admissible e-processes and martingales coincide.
			
			\subsection{Invariance under a sequence of groups}\label{sec:sequential_data}
				We embed the sequential setting in a latent sample space $\mathcal{X}$.
				In particular, we assume we have a nested sequence of subspaces $(\mathcal{X}^n)_{n\geq 0}$ of $\mathcal{X}$:  $\mathcal{X}^{n} \subseteq \mathcal{X}^{n+1}$, which are tied together through a sequence of continuous maps $(\textnormal{proj}_{\mathcal{X}^n})_{n \geq 0}$ which project onto the subsets,  $\textnormal{proj}_{\mathcal{X}^n} : \mathcal{X} \to \mathcal{X}^n$.
				With a projection map, we mean that such a map satisfies $\textnormal{proj}_{\mathcal{X}^n}(x) = x$ if $x \in \mathcal{X}^n$, and we assume they are compatible: $\textnormal{proj}_{\mathcal{X}^n} \circ \textnormal{proj}_{\mathcal{X}^{n+1}} = \textnormal{proj}_{\mathcal{X}^n}$, for all $n$.
				
				To describe the sequence of data we are to observe, suppose there is some latent random variable $X$ on $\mathcal{X}$, of which we sequentially observe an increasingly rich sequence $(X^n)_{n \geq 0}$ of projections $X^n = \textnormal{proj}_{\mathcal{X}^n}(X)$, $n \geq 0$.\footnote{This latent random variable is introduced for ease of exposition and it need not be modelled or `exist'.}
				This construction ensures that this sequence of random variables induces a filtration $(\sigma(X^n))_{n \geq 0}$.
				
				Next, we consider the group structure.
				Our sequential group structure is embedded into a (possibly non-compact) group $\mathcal{G}$ that acts continuously on $\mathcal{X}$.
				In particular, we consider a nested sequence of compact subgroups $(\mathcal{G}_n)_{n \geq 0}$ of $\mathcal{G}$.
				We assume the projection map induces a group action of $\mathcal{G}_n$ on $\mathcal{X}^n$ through the group action on $\mathcal{X}$: $Gx^n = \textnormal{proj}_{\mathcal{X}^n}(Gx)$, for all $G \in \mathcal{G}_n$, $x^n = \textnormal{proj}_{\mathcal{X}^n}(x)$, $x \in \mathcal{X}$.\footnote{This is well-defined if and only if $\text{proj}_{\mathcal{X}^n}(x^1) = \text{proj}_{\mathcal{X}^n}(x^2) \implies \textnormal{proj}_{\mathcal{X}^n}(Gx^1) = \textnormal{proj}_{\mathcal{X}^n}(Gx^2)$ for all $G \in \mathcal{G}_n$ and $x^1, x^2 \in \mathcal{X}$ (see, for example, Theorem 2.4 in \cite{eaton1989group}).}
				This assumption ensures we can use the groups $(\mathcal{G}_n)_{n \geq 0}$ and observations $(X^n)_{n \geq 0}$ without reference to the latent $\mathcal{G}$, $\mathcal{X}$ and $X$.
				
				Our goal now is to test the hypothesis that $(X^n)_{n\geq 0}$ is invariant under $(\mathcal{G}_n)_{n \geq 0}$.
				If there is a terminal observation $X$ and group $\mathcal{G}$ then this is equivalent to testing $\mathcal{G}$ invariance of $X$.

				\begin{exm}[Exchangeability and i.i.d.]\label{exm:sequential_exchangeability}
					Suppose that $X^n = (Y_0, \dots, Y_n)$ for each $n$.
					Let us choose $\mathcal{G}_n = \mathfrak{P}_n$ as the group of permutations on $n+1$ elements.
					We say $(X^n)_{n \geq 0}$ is exchangeable if it is invariant under $(\mathfrak{P}_n)_{n \geq 0}$.
					Referring back to Remark \ref{rmk:ergodic}, testing exchangeability is equivalent to testing whether the sequence is i.i.d. by de Finetti's theorem.
				\end{exm}			
				\begin{exm}[Within-batch exchangeability]\label{exm:within_batch_exchangeability}
					Suppose we sequentially observe potentially unequally sized batches of data $Y_0, Y_1, \dots$, where each $Y_i$ is exchangeable, $i = 0, 1, \dots$.
					We can choose $\mathcal{G}_n = \mathfrak{P}^0 \times \mathfrak{P}^1 \times \cdots \times \mathfrak{P}^n$, where $\mathfrak{P}^i$ is the group of permutations acting on the batch $Y_i$.
					Defining $X^n = (Y_0, \dots Y_n)$, within-batch exchangeability can be viewed as invariance of $(X^n)_{n \geq 0}$ under this group $(\mathcal{G}_n)_{n \geq 0}$.
					
					If we view the elements of a batch as individual observations, then within-batch exchangeability is weaker than exchangeability of individual observations: we exclude permutations that swap observations across batches.
					Specifically, the groups we consider here are subgroups of the permutations on the set of the individual observations.
					The idea to test sequential invariance of all observations by batching units into pairs has been independently explored by \citet{saha2024testing}.
				\end{exm}
				
			\subsection{Filtration}
				To sequentially test $(\mathcal{G}_n)_{n \geq 0}$ invariance of $(X^n)_{n \geq 0}$, we require the sequence of groups to be predictable.
				Recall from Corollary \ref{cor:conditional_validity} that validity under $\mathcal G_n$ is equivalent to validity under $\mathrm{Unif}(O_{X^n})$ conditional on the orbit.
				For this reason, we may work with the coarsened filtration
				\begin{align*}
					 \sigma(O_{X^0}) \subseteq \sigma(X^0, O_{X^1}) \subseteq \sigma(X^1, O_{X^2}) \subseteq \sigma(X^2, O_{X^3}) \subseteq \cdots,
				\end{align*}
				which keeps exactly the information needed to formulate a conditional e-value at each step.
				We use the shorthand $\mathcal{I}_n = \sigma(X^n, O_{X^{n+1}})$, $n \geq 0$ and $\mathcal{I}_0 = \sigma(O_{X^0})$.

			\subsection{Test martingale}
				To construct a test martingale, we construct conditional e-values $\varepsilon_n$, that are valid for $\mathcal{G}_n$ invariance conditional on the past data:
				\begin{align}\label{eq:conditional_e-value}
					\mathbb{E}^{\textnormal{Unif}(O_{X^n})}[\varepsilon_n \mid \mathcal{I}_{n-1}] \leq 1.
				\end{align}
				The test martingale itself is then given by its running product: $\varepsilon^n
						= \prod_{i = 0}^n \varepsilon_i$.
				
				In Proposition \ref{prp:conditional_e-value}, we present a characterization of the conditional e-value \eqref{eq:conditional_e-value}.
				The trick underlying this result is captured in Lemma \ref{lem:adapted_distribution}, which characterizes the conditional distribution $\textnormal{Unif}(O_{X^n}) \mid (X^{n-1}, O_{X^n})$ by means of a subgroup that stabilizes the past data.
				In particular, given $X^{n-1} = x^{n-1}$, we define $\mathcal{K}_n(x^{n-1})
						= \{G \in \mathcal{G}_n : Gx^{n-1} = x^{n-1}\}$, 	for $n \geq 1$ and $\mathcal{K}_0 = \mathcal{G}_0$.
				In Section \ref{proof:adapted_distribution} in the Supplementary Material, we show that this is indeed a compact subgroup of $\mathcal{G}_n$, and include a proof of a more general result.

				\begin{prp}\label{prp:conditional_e-value}
					An e-value $\varepsilon_n$ is conditionally valid given $O_{X^n}$ and $X^{n-1}$ if $\mathbb{E}_{\overline{K}_n}[\varepsilon_n(\overline{K}_nx^n)] \leq 1, \textnormal{ for every } x^n \in \mathcal{X}^n$.
				\end{prp}
				
				\begin{lem}\label{lem:adapted_distribution}
					Let $X^n$ be $\mathcal{G}_n$ invariant.
					Pick $O \in \mathcal{X}^n / \mathcal{G}_n$, $x^n \in O$,	 with $\textnormal{proj}_{\mathcal{X}^{n-1}}(x^n) = x^{n-1}$.
					Let $\overline{K}_n \sim \textnormal{Unif}(\mathcal{K}_n(x^{n-1}))$.
					Then, $X^n \mid (O_{X^n} = O, X^{n-1} = x^{n-1}) \overset{d}{=} \overline{K}_n x^{n}$.
				\end{lem}
				
				Proposition \ref{prp:conditional_e-value} shows that we may reduce the problem of constructing a conditional e-value to constructing an unconditional e-value that is valid for invariance under a data-dependent group $\mathcal{K}_n(X^{n-1})$.
				This means we may immediately apply the machinery derived in Section \ref{sec:ph_gi} and \ref{sec:optimal}, where we study the construction of such unconditional e-values.
				For example, following Section \ref{sec:ph_gi}, we may choose
				\begin{align*}
					\varepsilon_n(X^n)
						= \frac{T_n(X^n)}{\mathbb{E}_{\overline{K}_n}T_n(\overline{K}_nX^n)},
				\end{align*}
				where $T_n$ is a predictable non-negative test statistic.
				
				Alternatively, given a predictable alternative $\mathbb{Q}_n$ on $\mathcal{X}^n$, we may define $\overline{\mathbb{Q}}_n = \mathbb{E}_{\overline{K}_n}[\overline{K}_n\mathbb{Q}_n]$ with densities $q_n$ and $\overline{q}_n$ with respect to some reference measure and construct an expected utility-optimal e-value as in Section \ref{sec:optimal} based on $d\mathbb{Q}_n/d\overline{\mathbb{Q}}_n$.
				
				\begin{rmk}\label{rmk:induced_invariance}
					With the test martingale, we are effectively testing whether $(X^n)_{n \geq 1}$ is $(\mathcal{K}_n)_{n \geq 0}$ invariant.
					As the subgroups $(\mathcal{K}_n)_{n \geq 0}$ may be less rich than the original groups, we are testing a larger hypothesis than $(\mathcal{G}_n)_{n \geq 0}$ invariance.
					This shows where the test martingale loses power compared to an e-process.
					We illustrate this in Example \ref{exm:sequential_sphericity}, \ref{exm:seq_exch_degenerate} and \ref{exm:within_batch_exch}.
				\end{rmk}			
				
				\begin{exm}[Sequential sphericity]\label{exm:sequential_sphericity}
					Suppose that $\mathcal{X}^n = \mathbb{R}^n$ so that $X^n$ is a random $n$-vector for all $n$.
					Let $\mathcal{O}_n$ be the collection of $n \times n$ orthonormal matrices.
					Then, $X^n$ is said to be spherically distributed if it is invariant under $\mathcal{O}_n$.
					We consider testing invariance of the sequence $(X^n)_{n \geq 1}$ under matrix multiplication by the orthonormal matrices in $(\mathcal{O}_n)_{n \geq 1}$.
					
					In this example, the orbit $O_{X^n}$ is the hypersphere in $n$ dimensions that contains $X^n$.
					As a consequence, the effective filtration reveals the previous observations $X^{n-1}$ and the length of $X^n$.
					Together, these determine $X^n$ up to the sign of its final element.
					As a result, $\mathcal{K}_n$ contains two elements: $\textnormal{diag}(1, \dots, 1, 1)$ and $\textnormal{diag}(1, \dots, 1, -1)$, which flips the sign of the final element.
					This is equivalent to testing whether $X^n$ is invariant under sign-flips.
				\end{exm}
				
				\begin{exm}[Test martingale for exchangeability]\label{exm:seq_exch_degenerate}
					Continuing from Example \ref{exm:sequential_exchangeability}, suppose we sequentially observe $X^n = (Y_0, Y_1, \dots, Y_n)$ that are exchangeable.
					
					Here, it turns out that $X^n$ is degenerate conditional on $\sigma(X^{n-1}, O_{X^n})$.
					In particular, $X^{n-1} = (Y_0, Y_1, \dots, Y_{n-1})$ and $O_{X^n}$ equals the multiset $\{Y_0, \dots, Y_{n}\}$.
					Hence, $Y_n$ is simply the value in $O_{X^n}$ that is not accounted for in $X^{n-1}$.
					As a consequence, the conditional distribution $X^{n}$ given $X^{n-1}$ and $O_{X^n}$ is degenerate.
					Assuming the realizations are distinct, this means $\mathcal{K}_n$ only contains the identity element for each $n$.
					
					A consequence is that it is impossible to sequentially test exchangeability with a test martingale under the filtration $(\sigma(X^{n}))_{n \geq 0}$, as previously observed by \cite{vovk2021testing} and \cite{ramdas2022testing}.
					Our discussion gives some context around their impossibility result, by showing it may be interpreted as the group $\mathcal{K}_n$ becoming degenerate.
				\end{exm}
	
				\begin{exm}[Test martingale for within-batch exchangeability]\label{exm:within_batch_exch}
					Continuing from Example \ref{exm:within_batch_exchangeability}, let us again consider $X^n = (Y_0, \dots, Y_n)$, where each $Y_i$ is an exchangeable batch of data.
					Let us assume the realizations are distinct in each batch.
					Then, $\mathcal{K}_n(X^{n-1}) = \{I^0\} \times \{I^1\} \times \cdots \times \{I^{n-1}\} \times \mathfrak{P}^n$, where $I^i$ denotes the identity permutation acting on the $i$th batch, for $n \geq 1$ and $\mathcal{K}_0 = \mathfrak{P}^0$.
					That is, the conditional distribution of $X^n$ is uniform on the final batch.
					Interestingly, the stabilizer $\mathcal{K}_n(X^{n-1})$ does not depend on $X^{n-1}$.
					
					As discussed in Example \ref{exm:within_batch_exchangeability}, exchangeability implies within-batch exchangeability.
					This means rejecting within-batch exchangeability also rejects exchangeability.
					As a result, we can construct a sequential test for exchangeability by merging observations into batches.
					This of course impoverishes the filtration, since we only look at the data after a batch has arrived.
					The size of a batch is allowed to be adaptive.
					This reasoning is generalized in Appendix \ref{sec:impoverishing}.
				\end{exm}

		\section{Simulations and application}
						
			\subsection{Case-control experiment and learning the alternative}
				In this simulation study, we consider a hypothetical case-control experiment in which units are assigned to either the treated or control set uniformly at random.
				In each interval of time, we receive the outcomes of a number of treated and control units, where the number of treated and control units is Poisson distributed with parameter $\theta > 0$ and a minimum of 1.
				The outcomes of the treated units are $\mathcal{N}(a, 1)$-distributed and the outcomes of the controls are $\mathcal{N}(b, 1)$-distributed.
				The true mean and variance are considered unknown, and are adaptively learned based on the previously arrived data.
				As a batch of data, we consider the combined observations of both the treated and control units that arrived in the previous interval of time.
				
				As a result, a batch $X_t$ of $n^t$ outcomes, consisting of $n_a^t$ treated and $n_b^t$ control units, can be represented as
				\begin{align*}
					X_t
						\sim 	
						\begin{bmatrix}
							1_{n_{a}^t}a\\
							1_{n_{b}^t}b
						\end{bmatrix}
						+ \mathcal{N}\left(0, I\right),
				\end{align*}
				where $1_{n_a^t}$ and $1_{n_b^t}$ denote vectors of $n_a^t$ and $n_b^t$ ones, respectively, and the first $n_a^t$ elements correspond to the treated units, without loss of generality.
				We would like to base our test statistic on the difference of sample means: $\overline{1}_{n^t}'X_t \sim \mathcal{N}\left(a - b,  1/n_a^t + 1/n_b^t\right)$,
				where $\overline{1}_{n^t} = (1_{n_{a}^t} (n_a^t)^{-1}, -1_{n_{b}^t}(n_b^t)^{-1})$.
				In particular, we will test the null hypothesis that the elements of a batch $X_t$ are exchangeable and so $a = b$, against the alternative hypothesis that $a > b$.
				
				We use a test martingale based on the log-optimal e-value for testing exchangeability against our current estimate of the Gaussian alternative, as derived in Appendix \ref{sec:gaussian},
				\begin{align*}
					\varepsilon_t = \frac{\exp\{(\widehat{a}_{t-1} - \widehat{b}_{t-1})/ \widehat{\sigma}_{t-1}^2 \times \overline{1}_{n^t}'X_t\}}{\mathbb{E}_{\overline{G}} \exp\{(\widehat{a}_{t-1} - \widehat{b}_{t-1})/ \widehat{\sigma}_{t-1}^2 \times \overline{1}_{n^t}'\overline{G}X_t\}},
				\end{align*}
				where $\widehat{a}_{t-1} - \widehat{b}_{t-1} = \overline{1}_{n^{t-1}}'X_{t-1}$ is our treatment estimator at time $t-1$ and $\widehat{\sigma}_{t-1}^2$ is its pooled sample variance estimator, and $\overline{G}$ is uniform on the permutations of $n^t$ elements.
				For the first batch, we can either rely on an educated guess, or skip it for inference and only use it for estimating these parameters.
				We estimate the normalization constant by using 100 permutations drawn uniformly at random with replacement.
				
				For our simulations, we consider the arrival of 40 batches with $\theta = 25$.
				Without loss of generality, we choose $a = b = 0$ under the null, and $a = .2$ and $b = 0$ under the alternative.
				To use in the first batch, we choose $\widehat{a}_0 = .2$, $\widehat{b}_0 = 0$ and $\widehat{\sigma}_0^2 = 1$.
				
				In Figure \ref{fig:power_cc}, we plot the test-martingale-based e-processes for 1\,000 simulations.
				The dotted line indicates the value $20 = 1/0.05$, so that exceeding this line corresponds to a rejection at level $\alpha = 0.05$.
				The plot on the left features the setting under the null, and the plot on the right the setting under the alternative.
				To make the figure easier to interpret, we plot at each time the line below which 5\%, 50\% and 95\% of the test martingales have remained up until that point.
				For example, in the right plot, roughly 95\% of the e-processes have exceeded 20 at batch 23, so that the power at level $\alpha$ is roughly 95\% after 23 batches.
				As expected, the left plot shows that 95\% of the e-processes remain below 20 under the null.
				
				\begin{figure}
					\center
					\includegraphics[width=7cm]{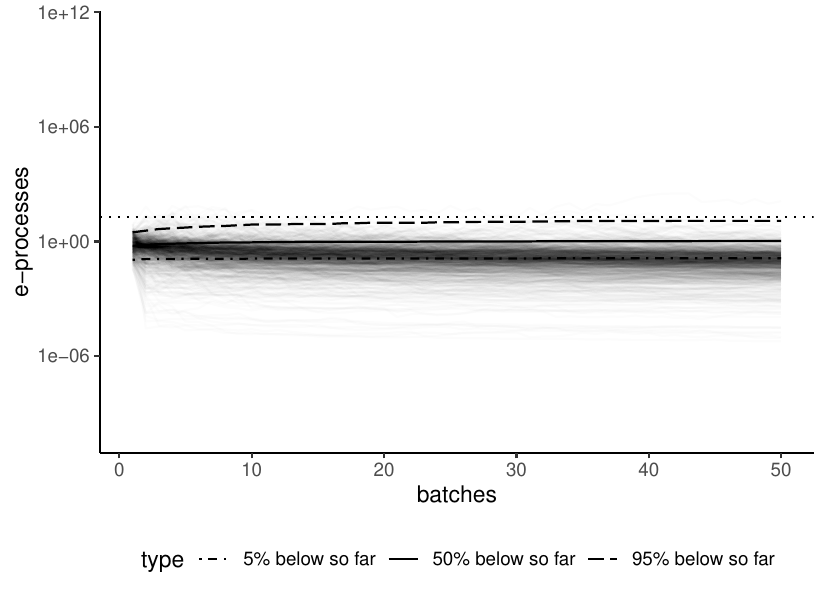}
					\includegraphics[width=7cm]{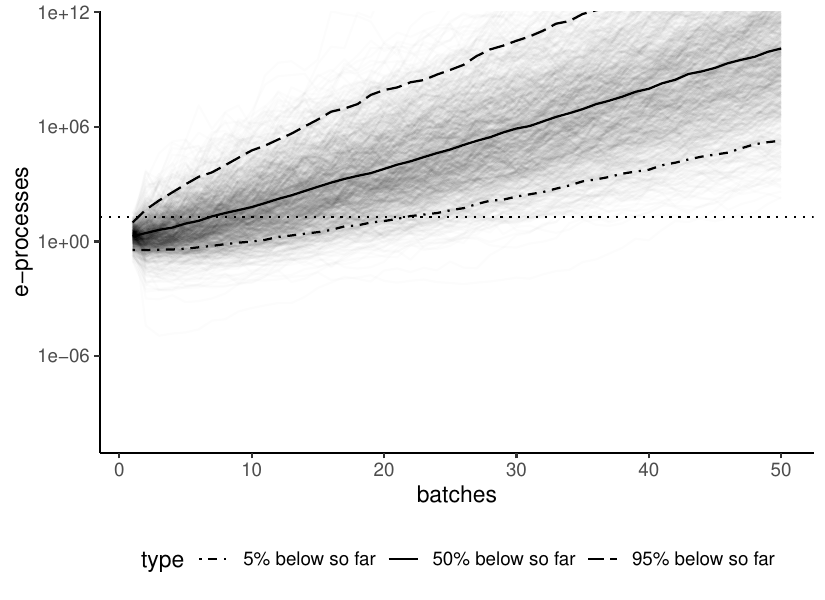}
					\caption{Plots of 1\,000 e-processes over the number of arrived batches. The highlighted lines are running quantiles: x\% of the e-processes have not crossed above the line at the indicated time. The plot on the left is under the null hypothesis, and the plot on the right is under the alternative. The horizontal dotted line is at 20.}
					\label{fig:power_cc}
				\end{figure}
				
		\subsection{Testing symmetry and comparison to \cite{de1999general}}
			In this simulation study, we consider testing sign-symmetry of data as in Appendix \ref{sec:symmetry}.
			We compare our e-process to the one based on \cite{de1999general} against the alternative $X_i \sim \mathcal{N}(m, 1)$ with $m = 1$.
			
			We plot 1\,000 e-processes of each type in Figure \ref{fig:symmetry}.
			The plot on the left is our log-optimal e-value-based e-process, whereas the plot on the right is based on \cite{de1999general}.
			The figure shows that our e-processes grow much more quickly.
			This coincides with the observation made by \cite{ramdas2022admissible} that the e-process based on \cite{de1999general} is inadmissible.			
			
			\begin{figure}
				\includegraphics[width=7cm]{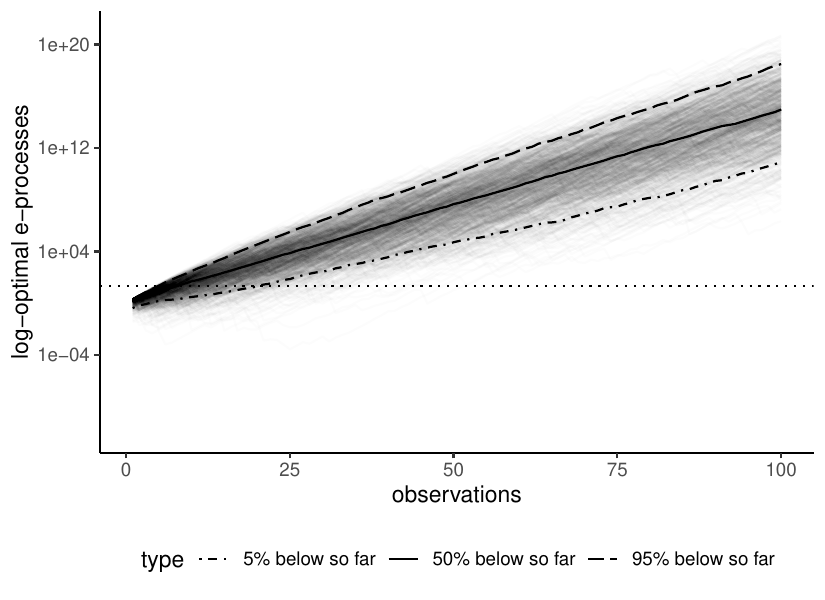}
				\includegraphics[width=7cm]{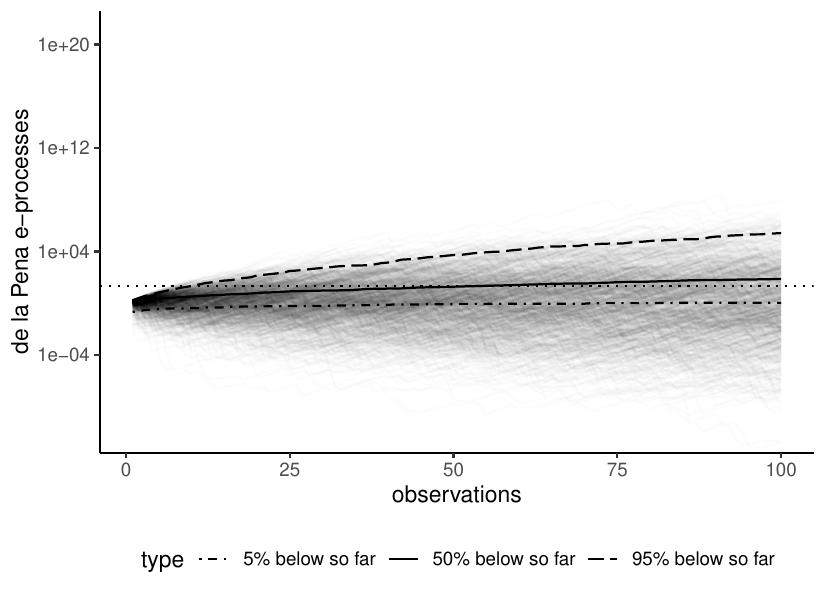}
				\caption{Plots of 1\,000 e-processes over the number of arrived observations under a normal alternative with mean $m = 1$. The highlighted lines are running quantiles: x\% of the e-processes have not crossed above the line at the indicated time. The plot on the left is for our log-optimal e-value-based e-process, and the plot on the right is for the one based on \cite{de1999general}. The horizontal dotted line is at 20.}
				\label{fig:symmetry}
			\end{figure}
		
		\subsection{Illustrative application: optimal e-values for the hot hand}\label{sec:hot_hand}
			Following Example \ref{exm:hot_hand}, we apply our methodology to test the hot hand in basketball.
			Here, we observe the outcomes (hits/misses) of $n$ shots of a basketball player and want to test whether they are exchangeable.
			
			To specify an e-value, we may condition on (\#hit, \#miss), which is equivalent to conditioning on the orbit.
			Following Corollary \ref{cor:uniform_over_marginals}, we specify the alternative on the orbit, to describe how we believe the hot hand works.
			This absolves us from having to specify an alternative across orbits, which would require prior knowledge of the skills of the shooters.
			
			For simplicity, we say a player is `hot' if they hit $k$ shots in a row.
			If a player is hot, suppose this boosts their probability of hitting the next shot conditional on the orbit through $p_{\textnormal{hot}} = (p_{\textnormal{neutral}})^\beta $, 	where $p_{\textnormal{neutral}}$ represents the conditional probability to hit in the absence of a hot hand: the number of remaining hits divided by the number of remaining shots in the sequence.
			This means that if $\beta = 1$, $p_{\textnormal{hot}} = p_{\textnormal{neutral}}$, and $p_{\textnormal{hot}} > p_{\textnormal{neutral}}$ when $\beta < 1$.
			For example, if $p_{\textnormal{neutral}} = 0.5$ and $\beta = 0.9$, then $p_{\textnormal{hot}} \approx 0.536$ --- a modest boost.
			
			For example, suppose the shot sequence is 111010, where 1 represents a hit and 0 a miss, and let $\beta = 0.9$ and $k = 2$.
			As there are $\binom{6}{4} = 15$ permutations of this sequence, the conditional probability of this sequence given the orbit equals 1/15 under the null. 
			Under the alternative, we decompose the conditional probability of the shot sequence given (\#hit, \#miss), into a sequence of further conditional probabilities given the previous shot outcomes:
			\begin{align*}
				&\textnormal{Pr}_{\beta = 0.9}( 111010 \mid (4, 2)) \\
					&= \textnormal{Pr}_{\beta = 0.9}(\textnormal{shot }1 = 1 \mid (4, 2)) \times \textnormal{Pr}_{\beta = 0.9}(\textnormal{shot }2 = 1 \mid (4, 2), \textnormal{shot }1 = 1) \times \cdots \\
					&= 4/6 \times 3/5 \times (2/4)^{0.9} \times (1 - (1/3)^{0.9}) \times 1/2 \times 1 \approx 0.0673.
			\end{align*}
			 where the powers of $0.9$ are because the preceding two shots were a hit, increasing the probability of a subsequent hit.
			By Corollary \ref{cor:log-optimal}, the resulting log-optimal e-value equals $\approx 0.0673 / (1/15) = 1.0095$ --- tiny evidence against no hot hand.
			This is unsurprising, as \citet{ritzwoller2022uncertainty} show long sequences are required to detect a hot hand.
			
			We offer an alternative solution by leveraging the merging properties of e-values: the product of independent e-values is also an e-value.
			We apply this idea to the controlled shooting experiment data collected by \citet{gilovich1985hot}, with 26 shooters taking up to 100 shots each.\footnote{We retrieved this data from the Supplementary Material of \citet{miller2018surprised}.}
			We consider variations of the hot hand that trigger after 1, 2 or 3 consecutive hits, with $\beta \in \{0.85, 0.9\}$.
			
			Table \ref{tab:hot_hand} reports the product of the e-values for the individual shooters.
			To interpret these e-values, recall that their reciprocals $\mathfrak{p} = 1/e$ are post-hoc p-values, which we may interpret as a rejection at level $\mathfrak{p}$ under a generalized Type-I error \citep{koning2023markov, grunwald2023beyond}.
			Looking at the product, we find strong evidence in support of the null (no hot hand) when compared to a 1-hit hot hand, but we find substantial evidence against the null for 2-hit and 3-hit triggers.
			The full table with e-values for each shooter is reported in Appendix \ref{appn:hot_hand}.
			
			\begin{table}
				\begin{tabular}{lrrrrrr}
				\toprule
				\multicolumn{1}{c}{Trigger} & \multicolumn{2}{c}{1 hit} & \multicolumn{2}{c}{2 hits} & \multicolumn{2}{c}{3 hits} \\
				\cmidrule(l{3pt}r{3pt}){2-3} \cmidrule(l{3pt}r{3pt}){4-5} \cmidrule(l{3pt}r{3pt}){6-7}
				\multicolumn{1}{c}{$\beta$} & 0.85 & 0.90 & 0.85 & 0.90 & 0.85 & 0.90\\
				\midrule
				\addlinespace
				Product e-value & 0.007 & 0.180 & 3.108 & 4.460 & 7.489 & 5.525\\
				Post-hoc p-value & 142.9 & 5.556 & 0.322 & 0.224 & 0.134 & 0.181 \\
				\bottomrule
				\end{tabular}
				\caption{Product of log-optimal e-values and post-hoc p-values ($\mathfrak{p} = 1/e$) for the controlled shooting experiment of \citet{gilovich1985hot} for exchangeability against several hot hand alternatives, triggering after 1-3 hits for a modest effect $(\beta = 0.85)$ and weak effect $(\beta = 0.9)$. }
				\label{tab:hot_hand}
			\end{table}

\section*{Acknowledgments}
	We thank Sam van Meer, Muriel P\'erez-Ortiz, Tyron Lardy, Will Hartog, Yaniv Romano, Jake Soloff, Stan Koobs, Peter Gr\"unwald, Wouter Koolen, Aaditya Ramdas and Dante de Roos for useful discussions and comments.
	Part of this research was performed while the author was visiting the Institute for Mathematical and Statistical Innovation, which is supported by the National Science Foundation (Grant No. DMS-1929348).
\section*{Funding}
	The author is supported by a Starter Grant of the Dutch government.
\bibliographystyle{plainnat}
\bibliography{bibliography}

\clearpage
\section*{Supplementary Material}

	\section{Illustration: optimal e-values for invariance against Gaussian location-shift}\label{sec:gaussian}
		In this section, we illustrate our optimal e-values for testing invariance under a group of orthonormal matrices, against a Gaussian alternative under a location shift.
		If we include all orthonormal matrices, this yields clean connections to parametric theory and Student's $t$-test.
		Moreover, we also consider exchangeability, which reveals an interesting relationship to the softmax function.
		In addition, we consider sign-symmetry, which we relate to a previously-studied e-value based on \cite{de1999general}, and to work of \citet{vovk2024nonparametric}.
		
		We start with an exposition of the invariance-based concepts for the orthogonal group $\mathcal{O}(d)$ that consists of all orthonormal matrices.

		\subsection{Sphericity}\label{sec:sphericity_background}
			Suppose that $\mathcal{Y} = \mathbb{R}^d\setminus\{0\}$ and $\mathcal{G} = \mathcal{O}(d)$ is the orthogonal group, which can be represented as the collection of all $d \times d$ orthonormal matrices, $d \geq 1$.
			The orbits $O_y = \{z \in \mathcal{Y}\ |\ z = Gy, \exists G \in \mathcal{G}\}$ of $\mathcal{G}$ in $\mathbb{R}^d$ are the concentric $d$-dimensional hyperspheres about the origin.
			Each of these hyperspheres can be uniquely identified with their radius $\mu > 0$.
			To obtain a $\mathcal{Y}$-valued orbit representative, we multiply $\mu$ by an arbitrary unit $d$-vector $\iota$ to obtain $\mu\iota$.
			For example $y$ lies on the orbit $O_y$ that is the $d$-dimensional hypersphere with radius $\|y\|_2$, and has orbit representative $[y] = \|y\|_2\iota$.
											
			For simplicity, we now first focus on the subgroup $SO(2)$ of $O(2)$ and its action on $\mathbb{R}^2\setminus\{0\}$, which exactly describes the (orientation-preserving) rotations of the circle, and has the same orbits as $O(2)$.
			The reason we focus on $SO(2)$, is because its group acts freely on each concentric circle.
			As a consequence, every element in the group can be uniquely identified with an element on the unit circle $S^1$ (and in fact on every orbit).
			We choose to identify the identity element with $\iota$, and we identify every element of $SO(2)$ with the element on the circle that we obtain if that rotation is applied to $\iota$.
			We denote this induced group action of the unit circle $S^1$ on $\mathcal{Y}$ by $\circ$.
							
			Under this bijection between the group and the unit circle, we can define our inversion kernel map $\gamma$ as $\gamma(y) = y/\|y\|_2$, which may be viewed as the group element that rotates $\iota$ to $y/\|y\|_2$.
			To see that $\gamma$ is indeed an inversion kernel, observe that 
			\begin{align}\label{exm:eq:gamma}
				\gamma(y)[y] 
					= y/\|y\|_2 \circ \iota\|y\|_2
					= [(y/\|y\|_2) \circ \iota]\|y\|_2 
					= (y/\|y\|_2)\|y\|_2  = y,
			\end{align}
			where the second equality follows from the fact that the action of $(y/\|y\|_2)$ on $\iota$, rotates $\iota$ to $y/\|y\|_2$.
			Invariance of a $\mathcal{Y}$-valued random variable $Y$ under $\mathcal{G}$, also known as sphericity, can then be formulated as `$\gamma(Y)$ is uniform on $S^1$'. 
			
			For $\mathcal{O}(2)$ or the general $d > 2$ case, the group action is no longer free on each orbit.
			As a result there may be multiple group actions that carry $\iota\|y\|_2$ to a point $y$ on the hypersphere.
			While this may superficially seem like a potentially serious issue, we may simply view $\gamma(y)$ as uniformly drawn from all the `rotations' that carry $\iota\|y\|_2$ to $y$.
			As a result, the only difference is that \eqref{exm:eq:gamma} will now hold almost surely, which suffices for our purposes.
						
		\subsection{Neyman--Pearson optimal e-values for $\mathbb{Q}$ on the sample space: the $t$-test and its generalizations}\label{sec:LR_spherical_Y}
			Suppose that $Y \sim \mathcal{N}_d(\mu\iota, I)$ on $\mathbb{R}^d \setminus \{0\}$, $\mu > 0$ under the alternative and $Y$ is $\mathcal{G}$ invariant under the null.
			Let $\mathcal{G}$ be some compact group of orthonormal matrices; a subgroup of $\mathcal{O}(d)$.
							
			Here, we conveniently have the Lebesgue measure as a $\mathcal{G}$ invariant reference measure, so that we may apply Proposition \ref{prp:invariant_reference}, which means we only need to consider the classical Gaussian density with respect to the Lebesgue measure when deriving optimal e-values:
			\begin{align*}
				q(y) := 1/(2\pi)^{d/2}\exp\left\{-\frac{1}{2}\|y - \iota\mu\|_2^2\right\}.
			\end{align*}
			
			By Corollary \ref{cor:lehman-stein}, the Neyman--Pearson optimal test rejects at level $\alpha$ when
			\begin{align*}
				1/(2\pi)^{d/2}\exp\left\{-\frac{1}{2}\|y - \iota\mu\|_2^2\right\} &> q_{\alpha}^{\overline{G}}\left(1/(2\pi)^{d/2}\exp\left\{-\frac{1}{2}\|\overline{G}y - \iota\mu\|_2^2\right\}\right),
			\end{align*}
			where $\overline{G}$ is uniformly distributed on $\mathcal{G}$.
			This is equivalent to
			\begin{align*}
				-y'y + 2\mu\iota'y - \mu^2 &> q_{\alpha}^{\overline{G}}\left(-y'y + 2\mu\iota'\overline{G}y - \mu^2\right)
			\end{align*}
			so that the Neyman--Pearson optimal e-value / test may be concisely written as
			\begin{align}\label{ineq:t-test_correlation}
				\varepsilon^{\textnormal{NP}}
					= 1/\alpha \times \mathbb{I}\{\iota'y &> q_{\alpha}^{\overline{G}}\left(\iota'\overline{G}y\right)\},
			\end{align}
			which is independent of $\mu$, so that this test is optimal against $\mathcal{N}(\iota\mu, I)$, uniformly in $\mu$.
			
			\begin{rmk}
				For $\mathcal{G} = \mathcal{O}(d)$, the test \eqref{ineq:t-test_correlation} is equal to the $t$-test by Theorem 6 in \cite{koning2023more}.
				This matches the discussion in the final paragraphs of \citet{lehmann1949theory}, who also conclude that the $t$-test is uniformly most powerful for testing spherical invariance against $\mathcal{N}_d(\mu\iota, I)$, $\mu > 0$.
				
				If $\mathcal{G}$ is a subgroup of $\mathcal{O}(d)$, this test may be viewed as a generalization of the $t$-test under weaker conditions; see \citet{efron1969student} for an example in case of a group of sign-flips (diagonal matrices with diagonal elements in $\{-1, 1\}$).
				Our results here show that the approach by \citet{efron1969student} is most powerful for this sign-flipping group against Gaussianity.
			\end{rmk}

			\begin{rmk}[Optimality of the $t$-test beyond \citet{lehmann1949theory}]\label{rmk:beyond_LS}
				If $\mathcal{G} = \mathcal{O}(d)$, then the $t$-test \eqref{ineq:t-test_correlation} may be reformulated as
				\begin{align*}
					1/\alpha \times \mathbb{I}\{\iota'y / \|y\|_2 > q_{\alpha}^{\overline{G}}\left(\iota'\overline{G}\iota\right)\},
				\end{align*}
				as $q_{\alpha}^{\overline{G}}\left(\iota'\overline{G}y\right) 
						= q_{\alpha}^{\overline{G}}\left(\iota'\overline{G}\iota\|y\|_2\right)
						= \|y\|_2 q_{\alpha}^{\overline{G}}\left(\iota'\overline{G}\iota \right)$.
				Here, $\iota'y / \|y\|_2$ may be interpreted as the correlation coefficient between $\iota$ and $y$.
				
				Now, as the rejection event does not change if we apply a strictly increasing function to both sides, we may even conclude that the $t$-test is Neyman--Pearson-optimal for testing spherical invariance against \emph{any alternative with a density that is increasing in the correlation coefficient $\iota'y/\|y\|_2$.}
				This generalizes the result of \citet{lehmann1949theory}, who only conclude optimality against Gaussian location shifts.
			\end{rmk}

		\subsection{Log-optimal e-value}
			Following Corollary \ref{cor:log-optimal} and Proposition \ref{prp:invariant_reference}, the log-optimal e-value for $\mathcal{G}$-invariance against $\mathcal{N}(\iota\mu, I)$ is
			\begin{align}\label{eq:LR_spherical}
				\varepsilon^{\textnormal{log}}(y)
					= \frac{q(y)}{\mathbb{E}_{\overline{G}}[q(\overline{G}y)]}
					= \frac{\exp\left\{\mu y'\iota\right\}}{\mathbb{E}_{\overline{G}} \left[\exp\left\{\mu y'\overline{G}\iota\right\}\right]}.
			\end{align}
			While this may be viewed as the log-optimal version of the $t$-test, it is not uniformly log-optimal in $\mu$.
			
			If $\mathcal{G} = \mathcal{O}(d)$, it is also not uniformly log-optimal in the class of alternatives with densities increasing in $\iota'y/\|y\|_2$ as in Remark \ref{rmk:beyond_LS}.
			Echoing Example \ref{exm:e-value_t-test}, this underlines that there is no unique `e-value version' of the $t$-test, nor even a unique `log-optimal' version of the $t$-test: any e-value based on an alternative density $q$ that is non-decreasing in $\iota'y/\|y\|_2$ may qualify.
			The underlying `problem' is that the original $t$-test is Neyman--Pearson optimal uniformly against a large composite alternative, but specifying a log-optimal variant requires us to be much more specific about our alternative, because we cannot leverage the invariance of the e-value under	monotone transformations of the test statistic as in Remark \ref{rmk:beyond_LS}.

		\subsection{Alternative on orbits}
			We may apply the ideas in Section \ref{sec:optimal_orbit} to slightly enlarge the class of alternatives under which \eqref{eq:LR_spherical} is uniformly log-optimal by passing to the conditional distribution on each orbit.
			The conditional distribution of $Y \sim \mathcal{N}_d(\mu\iota, I)$ on each orbit is proportional to $\exp(\mu\iota'y)$, where $y$ is on the orbit with radius $\|y\|_2$.
			For $\|y\|_2 = 1$, this is also known as the von Mises-Fisher distribution.
			The log-optimal e-value on each orbit $\varepsilon_{|O}^{\textnormal{log}} : O \mapsto [0, \infty]$ indeed corresponds to $\varepsilon^{\textnormal{log}}$:
			\begin{align*}
				\varepsilon_{|O}^{\textnormal{log}}(y)
					= \frac{\exp\left\{\mu y'\iota\right\}}{\mathbb{E}_{\overline{G}} \left[\exp\left\{\mu y'\overline{G}\iota\right\}\right]}.
			\end{align*}	
			As a consequence $\varepsilon^{\textnormal{log}}$ is log-optimal against any mixture over such conditional distributions on orbits.

		\subsection{Alternative on $\mathcal{G}$}
			In this section, we reduce ourselves to $d=2$ and $SO(2)$, so that the group action is free and the group will be easy to represent.
			Following Section \ref{sec:sphericity_background}, we use a bijection between the unit circle $S^1$ and $SO(2)$ to more conveniently formulate the group using $S^1$.
			
			As an alternative on the group, we consider the projected normal distribution $\mathcal{P}\mathcal{N}_2(\mu\iota, I)$.
			This arises as the pushforward of the Gaussian through the inversion kernel: if $Y \sim \mathcal{N}_2(\mu\iota, I)$, then $\gamma(Y) = Y / \|Y\|_2 \sim \mathcal{P}\mathcal{N}_2(\mu\iota, I)$.
			Its density with respect to the uniform distribution on $S^1$ is
			\begin{align}\label{eq:projected_normal}
				\frac{\exp\{-\tfrac{1}{2}\mu^2\}}{2\pi}\left(1 + \mu \iota'v\frac{\Phi(\mu \iota'v)}{\phi(\mu \iota'v)}\right),
			\end{align}
			where $v \in S^1$, $\Phi$ is the normal cdf and $\phi$ the pdf (Presnell et al., 1998; Watson, 1983).
			For $\mu = 0$, this reduces to $1/(2\pi)$; the Haar-density.
			As a consequence, log-optimal e-value for testing the Haar measure against this projected normal distribution is simply the likelihood ratio between \eqref{eq:projected_normal} and $1/(2\pi)$:
			\begin{align*}
				\varepsilon_{S^1}^{\textnormal{log}}
					= \exp\{-\tfrac{1}{2}\mu^2\}\left(1 + \mu \iota'v\frac{\Phi(\mu \iota'v)}{\phi(\mu \iota'v)}\right).
			\end{align*}
			This may also be expressed as an e-value on $\mathcal{Y}$ by mapping through the inversion kernel:
			\begin{align*}
				\varepsilon_{S^1}^{\textnormal{log}}(\gamma(y))
					&= \exp\{-\tfrac{1}{2}\mu^2\}\left(1 + \mu \iota'\gamma(y)\frac{\Phi(\mu \iota'\gamma(y))}{\phi(\mu \iota'\gamma(y))}\right) \\
					&= \exp\{-\tfrac{1}{2}\mu^2\}\left(1 + \mu \iota'y / \|y\|_2\frac{\Phi(\mu \iota'y / \|y\|_2)}{\phi(\mu \iota'y / \|y\|_2)}\right)
			\end{align*}
			which is an increasing function in $\iota'y / \|y\|_2$ if $\mu > 0$. 
			
		\subsection{Permutations and softmax}\label{sec:softmax}
			The log-optimal e-value in \eqref{eq:LR_spherical} is strongly related to the softmax function.
			Indeed, if we choose $\overline{G}$ to be uniform on permutation matrices (which form a subgroup of the orthonormal matrices) and choose the unit vector $\iota = (1, 0, \dots, 0)$, then \eqref{eq:LR_spherical} becomes
			\begin{align}\label{eq:softmax}
				\frac{\exp\left\{\mu y_1\right\}}{\tfrac{1}{d}\sum_{i = 1}^d\exp\left\{\mu y_i\right\}}.
			\end{align}
			This is exactly the softmax function with `inverse temperature' $\mu \geq 0$.
			Hence, the softmax function can be viewed as a likelihood ratio statistic for testing exchangeability (permutation invariance) against $\mathcal{N}((\mu, 0, \dots, 0), I)$.
			
			\begin{rmk}\label{rmk:softrank}
				A related e-value appears in unpublished early manuscripts of \cite{wang2022false} and \cite{ignatiadis2023values}, who consider a `soft-rank' e-value of the type $\varepsilon_T$ as in \eqref{eq:exact_e_form} with the choice of statistic
				\begin{align}\label{eq:softrank_statistic}
					T(y)
						= \frac{\exp(\kappa y_1) - \exp(\kappa \min_j y_j)}{\kappa},
				\end{align}
				under exchangeability, for some inverse temperature $\kappa > 0$.
				
				Interestingly, this `soft-rank' e-value for $\kappa = \mu$ is larger than the softmax e-value \eqref{eq:softmax} if and only if the softmax e-value is larger than 1.
				In fact, the same holds if we replace $\exp(\kappa \min_j y_i)$ by any positive constant $c$, and the relationship flips if $c$ is negative.
				For a positive constant $c$, we would therefore expect the `soft-rank' e-value to be more volatile.
			\end{rmk}

		\subsection{Testing sign-symmetry}\label{sec:symmetry}
			Suppose $\mathcal{Y} = \mathbb{R}$ and $\mathcal{G} = \{-1, 1\}$.
			Then, invariance of $Y$ under $\mathcal{G}$ is also known as `symmetry' about 0, defined as $Y \overset{d}{=} -Y$.
			For testing symmetry against our normal location model with $\iota = 1$, the log-optimal e-value becomes
			\begin{align*}
				\exp\{\mu \iota'y\} / \mathbb{E}_{\overline{G}} \exp\{\mu\iota'\overline{G}y\}
					= 2\exp\{\mu y\} / \left[\exp\{\mu y\} + \exp\{-\mu y\}\right],
			\end{align*}
			This can be generalized to $\mathcal{Y} = \mathbb{R}^d$ and $\mathcal{G} = \{-1, 1\}^d$ and $\iota = d^{-1/2}(1, \dots, 1)'$.
			The log-optimal e-value becomes
			\begin{align}\label{e-value_sign}
				\exp\{\mu \iota'y\} / \mathbb{E}_{\overline{g}} \exp\{\mu\overline{g}'y\}
					= \prod_{i=1}^d \exp\{d^{-1/2}\mu y_i\} / \mathbb{E}_{\overline{g}_i} \exp\{d^{-1/2}\mu\overline{g}_iy_i\},
			\end{align}
			where $\overline{g}$ is a $d$-vector of i.i.d. Bernoulli distributed random variables on $\{-1, 1\}$ with probability .5.

			\begin{rmk}
				A related e-value can be derived from \cite{de1999general},
				\begin{align*}
					\exp\{Z - Z^2/2\}.
				\end{align*}
				This object can be connected to our likelihood ratio, by simply normalizing it by $\mathbb{E}_{\overline{g}}[\exp\{\overline{g}Z - (\overline{g}Z)^2/2\}]$:
				\begin{align*}
					\exp\{Z - Z^2/2\} / &\mathbb{E}_{\overline{g}}[\exp\{\overline{g}Z - (\overline{g}Z)^2/2\}] \\
						&= 
					2\exp\{Z - Z^2/2\} / \left[\exp\{-Z - Z^2/2\} + \exp\{Z - Z^2/2\}\right] \\
						&= 2\exp\{Z\} / \left[\exp\{-Z\} + \exp\{Z\}\right].
				\end{align*}
				This transformation makes the resulting e-value exact by Theorem \ref{thm:essentially_complete}, so that our e-value for sign-symmetry can be interpreted as an exact variant of the \cite{de1999general}-style e-value.
				This was also observed by \citet{vovk2024nonparametric}.
				
				Moreover, \cite{ramdas2022admissible} characterize the class of admissible e-processes for testing symmetry, and show that the e-process based on \cite{de1999general} is inadmissible.
				This inadmissibility is also visible in our simulations, where we find it is strongly dominated by ours. 
			\end{rmk}
			
			\begin{rmk}[Relationship to \citet{vovk2024nonparametric}]
				\citet{vovk2024nonparametric} also study the e-value \eqref{e-value_sign}.
				While they motivate this e-value from its reminiscence to the e-value in the Gaussian vs Gaussian setting, we show that it is in fact optimal for sign-symmetry a Gaussian location-shift.
				They also consider a particular sign-e-value, that relies on the number of positive signs.
				This may be viewed as mapping the data through the inversion kernel, and then constructing an e-value based on a particular statistic (the number of positive signs).
				A third e-value they consider relies on the number of ranks of observations with positive signs.
				This may be viewed as considering invariance under a group of both permutations and sign-flips, then mapping through the inversion kernel to the rank-sign combinations, and then deriving an e-value based on a particular statistic.
			\end{rmk}

	\section{Impoverishing filtrations}\label{sec:impoverishing}
			In the context of exchangeability, Example \ref{exm:seq_exch_degenerate} recovers the result that no powerful test martingales exist.
			Instead of passing to an e-process, as discussed in Section \ref{sec:e-processes} and Remark \ref{rmk:ergodic}, \citet{vovk2021testing} considers moving to a less-informative, `impoverished' filtration by passing to the ranks of the data.
			The practical implication is that we may no longer look at the full data, but only at the ranks.
			In exchange, it turns out that we may recover powerful martingales.
			
			In this section, we show how the impoverishment of a filtration works in the more general context of group invariance, for general statistics and for statistics that mimic the role that the ranks play in exchangeability.			
			This relies on two key ingredients: a subgroup $\mathcal{F} \subseteq \mathcal{G}$ and a statistic $H : \mathcal{X} \to \mathcal{Z}$ for which the subgroup induces a group action on its codomain.
			To induce a group action on the codomain, we require for each $x^1, x^2 \in \mathcal{X}$,
			\begin{align*}
				H(x^1) = H(x^2) \implies H(Fx^1) = H(Fx^2), \textnormal{ for all } F \in \mathcal{F}.
			\end{align*}
			Writing $z = H(x)$, the group action on the codomain is then defined as $Fz = H(Fx)$, $x \in \mathcal{X}$.
			This condition holds if and only if $H$ is equivariant under this group action: $FH(x) = H(Fx)$, for every $x \in \mathcal{X}$, $F \in \mathcal{F}$ \citep{eaton1989group}.
			For this reason, we will refer to it as an equivariant statistic.
			
			Proposition \ref{prp:equivariance} captures the key idea: measuring evidence against $\mathcal{F}$ invariance of $Z := H(X)$ also measures evidence against $\mathcal{G}$ invariance of $X$.
			As a consequence, we may apply all our methodology to testing $\mathcal{F}$ invariance of $Z$ and still obtain a valid e-value for $\mathcal{G}$ invariance of $X$.
			
			\begin{prp}\label{prp:equivariance}
				Let $\mathcal{F}$ be a subgroup of $\mathcal{G}$ and let $H : \mathcal{X} \to \mathcal{Z}$ be $\mathcal{F}$-equivariant.
				Then, an e-value that is valid for $\mathcal{F}$ invariance of $H(X)$ is also valid for $\mathcal{G}$ invariance of $X$.
			\end{prp}
			\begin{proof}
				$\mathcal{G}$ invariance of $X$ implies $\mathcal{F}$ invariance of $X$, as $\mathcal{F}$ is a subgroup of $\mathcal{G}$.
				Moreover, $\mathcal{F}$ invariance of $X$ implies $\mathcal{F}$ invariance of $H(X)$, as this group action is well-defined through the $\mathcal{F}$-equivariance of $H$.
				Hence, the hypothesis of $\mathcal{F}$ invariance of $H(X)$ is at least as large as that of $\mathcal{G}$ invariance of $X$.
				Hence, an e-value that is valid for the former is also valid for the latter.
			\end{proof}

			To apply this in the sequential context described in Section \ref{sec:sequential_data}, we must be careful to consider a sequence $(H_n)_{n \geq 0}$ of statistics that are appropriately glued together with a projection on its codomain $\mathcal{Z}^{n-1}$: $H_{n-1}(x^{n-1}) = \textnormal{proj}_{\mathcal{Z}^{n-1}}(H_n(x^{n}))$, and a nested sequence of subgroups $\mathcal{F}_n \subseteq \mathcal{G}_n$.
			In Example \ref{exm:reducing_to_batches}, we illustrate this approach by showing how we may reduce from exchangeability (continuing from Example \ref{exm:seq_exch_degenerate}) to within-batch exchangeability (continuing from Example \ref{exm:within_batch_exch}) by selecting a particular equivariant statistic.
			
			\begin{exm}[Reducing to within-batch-exchangeability]\label{exm:reducing_to_batches}
				Suppose $X^n = (Y_1, \dots, Y_n)$ and that $X^n$ is exchangeable.
				We now consider a statistic $H_n$ that effectively censors $X^n$ so that we only observe it in batches.
				Let $b_1, b_2, \dots$ denote the observation numbers at which a batch is completed, and $B_n$ the number of completed batches at time $n$.
				Then, we define the statistic equal to the most recently arrived batch $H_n(X^n) = X^{b_i}$, and its codomain $\mathcal{Z} = \mathcal{X}^{b_i}$, for all $b_i \leq n < b_{i +1}$, $i < B_n$.
				
				To induce a group action, we pass from the group $\mathfrak{P}_n$ of all permutations to its subgroup $\mathcal{F}_n = \mathfrak{P}^1 \times \mathfrak{P}^2 \times \cdots \times \mathfrak{P}^{B_n} \times I$, where $\mathfrak{P}^i$ permutes the observations within the $i$th batch of data, and $I$ acts as the identity on the yet-to-be-completed batch.
				
				It remains to verify that this indeed induces a group action.
				This means we need to verify $H_n(x_1^n) = H_n(x_2^n)$ implies $H_n(Fx_1^n) = H_n(Fx_2^n)$ for all $F \in \mathcal{F}_n$ and $x_1^n, x_2^n \in \mathcal{X}^n$.
				This is equivalent to checking whether $x_1^{b_i} = x_2^{b_i}$ implies $H_n(Fx_1^n) = H_n(Fx_2^n)$, where $b_i \leq n < b_{i+1}$, $i < B_n$.
				This is indeed satisfied, because $F$ only acts on the already completed batches.
			\end{exm}
			
		\subsection{Reduction to a single orbit}\label{sec:reduction_to_one_orbit}
			Recall from Section \ref{sec:e-processes} that admissible e-processes for group invariance may be viewed as measurable infimums over orbit-wise martingales.
			If there exists just a single orbit, then this infimum drops out so that admissible e-processes are martingales.
			While settings with just a single orbit may seem practically irrelevant, we can \emph{reduce} to a single orbit by finding an appropriate subgroup $\mathcal{F} \subseteq \mathcal{G}$ and accompanying $\mathcal{F}$-equivariant statistic $H$ such that $H(X)$ has only a single orbit under $\mathcal{F}$.
			
			While this approach may be applied to other statistics, we focus on a particularly attractive example of such a statistic: the unique inversion kernel $\gamma : \mathcal{X} \to \mathcal{G}$, which is an equivariant (possibly randomized) statistic \citep{kallenberg2017random}.
			Such an inversion kernel maps to the group, and the group acting on itself trivially has a single orbit: the group itself.
			Moreover, $\mathcal{G}$ is a subgroup of itself, so that Proposition \ref{prp:equivariance} applies, and we may measure evidence against $\mathcal{G}$ invariance of $X$ by measuring evidence against $\mathcal{G}$ invariance of $\gamma(X)$.
			
			By passing through such a statistic, we are effectively observing a draw $\widetilde{G} := \gamma(X)$ from the group itself.
			Recall from Section \ref{sec:optimal_on_group} that an e-value on the group, $\varepsilon : \mathcal{G} \to [0, \infty]$, is valid for $\mathcal{G}$ invariance if and only if
			\begin{align*}
				\mathbb{E}_{\overline{G}}[\varepsilon(\overline{G})] \leq 1.
			\end{align*}
			Analogously, an e-process $(\varepsilon_n)_{n \geq 0}$ is valid with respect to some filtration if and only if
			\begin{align*}
				\mathbb{E}_{\overline{G}}[\varepsilon_\tau(\overline{G})]
					\leq 1,
			\end{align*}	
			for every stopping time $\tau$ that is adapted to the same filtration.
			
			As mentioned, any admissible e-process for such a simple hypothesis is a martingale \citet{ramdas2022admissible}.
			This means such an admissible e-process may be induced as a Doob martingale in the style of \citet{koning2025sequentializing}, as discussed in Section \ref{sec:e-process from e-value} for an arbitrary filtration $(\mathcal{I}_n)_{n \geq 0}$ through			
			\begin{align*}
				\varepsilon_n = \mathbb{E}_{\overline{G}}[\varepsilon(\overline{G}) \mid \mathcal{I}_n].
			\end{align*}
			
			Instead of such a backwards-induction of a martingale, we may forwards-construct a martingale by imposing some additional structure as in Section \ref{sec:sequential_data}.
			In Section \ref{sec:sequential_data}, we assumed that we are to observe an increasingly rich sequence of data $(X^n)_{n \geq 0}$, $X^{n} = \textnormal{proj}_{\mathcal{X}^{n}}(X^{n+k})$, $n, k \geq 0$.
			Moreover, we introduced a nested sequence of groups $(\mathcal{G}_n)_{n \geq 0}$, $\mathcal{G}_n \subseteq \mathcal{G}_{n+1}$.
			These two sequences were made compatible with each other by assuming that each projection map $\textnormal{proj}_{\mathcal{X}^{n}}$ is $\mathcal{G}_n$-equivariant.
			We now additionally assume that the orbit representatives are chosen in a compatible manner: $[\textnormal{proj}_{\mathcal{X}^{n}}(x^{n+k})] = \textnormal{proj}_{\mathcal{X}^{n}}([x^{n+k}])$.
			This makes the inversion kernels compatible with each other, as it implies\footnote{Since $\gamma_n(x^n)[x^n] \overset{a.s.}{=} x^n = \textnormal{proj}_{\mathcal{X}^n}(x^{n+k}) \overset{a.s.}{=} \textnormal{proj}_{\mathcal{X}^n}(\gamma_{n+k}(x^{n+k})[x^{n+k}])$, which, if $\gamma_{n+k}(x^{n+k}) \in \mathcal{G}_n$ equals $\gamma_{n+k}(x^{n+k})\textnormal{proj}_{\mathcal{X}^n}([x^{n+k}]) = \gamma_{n+k}(x^{n+k})[\textnormal{proj}_{\mathcal{X}^n}(x^{n+k})]= \gamma_{n+k}(x^{n+k})[x^n]$.}
			\begin{align*}
				\gamma_{n+k}(x^{n+k}) \in \mathcal{G}_{n} \implies \gamma_{n}(x^{n})[x^{n}] \overset{a.s.}{=} \gamma_{n+k}(x^{n+k})[x^{n}],
			\end{align*}
			so that $\gamma_{n+k}$ is an inversion kernel for $\mathcal{G}^n$ acting on $\mathcal{X}^n$.
			Assuming the group action is free, this implies $\gamma_n(x) = \gamma_{n+k}(x)$, for $x \in \mathcal{X}^n$, by the uniqueness of the inversion kernel.
			If the group action is not free, $\gamma_{n+k}$ may be viewed as a randomized statistic, which then shares the same unique distribution as $\gamma_n$ on $\mathcal{G}^n$.
			
			We may now apply the machinery from Section \ref{sec:sequential_data} to derive a test martingale by constructing a conditional e-value for $\gamma_n(X^n)$ conditional on $\gamma_{n-1}(X^{n-1})$.
			We illustrate this process in Example \ref{exm:rank_inversion_kernel_sequential} and \ref{exm:sphere_inversion_kernel_sequential}.
			
			\begin{exm}[Exchangeability and ranks]\label{exm:rank_inversion_kernel_sequential}
				Suppose we have data $X^n = (Y_1, Y_2, \dots, Y_n)$ and $X^n$ is exchangeable for each $n$; invariant under the group $\mathcal{G}_n$ of permutations.
				This means $\mathcal{G}_{n-1}$ is a subgroup of $\mathcal{G}_n$.
				Next we must select some orbit representative, which in the case of exchangeability comes down to selecting some canonical order of the elements.
				If the elements are real-valued, or admit some other natural ordering, then it makes sense to sort the elements accordingly and use this as the orbit representative, but any ordering suffices.
				
				For example, suppose that $X^n = 7314$, which we use as a shorthand for $Y_1 = 7, Y_2 = 3, Y_3 = 1, Y_4 = 4$.
				Suppose the orbit representative is selected as $1347$, then $\textnormal{Rank}_n(X^n) = 4213$.
				The ranks $4213$ may be interpreted as encoding the permutation that instructs how the elements of the orbit representative $1347$ must be permuted in order to recover $X^n = 7314$: $4213$ states that the 4th element of $1347$ should be placed in the first position, the 2nd element of $1347$ in the second position, the 1st element in the third position and the 3rd element in the fourth position.
				That is, it encodes a permutation group action `$\times$':  $4213 \times 1347 = 7314$.
				This means $\textnormal{Rank}_n(X^n) \times [X^n] = X^n$, which is exactly the definition of an inversion kernel $\gamma_n = \textnormal{Rank}_n$.
				
				Now, let us consider the distribution of $\gamma_n(X^n) \mid \gamma_{n-1}(X^{n-1})$.
				Conditional on the ranking $\textnormal{Rank}_{n-1}(X^{n-1}) = 312$ of the first $(n - 1)$ elements, we have under exchangeability that the ranking of $n$ elements $\textnormal{Rank}_{n}(X^{n})$ is uniform on $\{3124, 4123, 4132, 4231\}$.
				Hence, constructing a conditional e-value is equivalent to constructing an e-value that is valid under a uniform distribution on this set.
				
				In the context of rank-based testing of exchangeability, it is common to focus on the rank of the most recently arrived element.
				Working out the above for each possible conditioning, it is straightforward to show that this `last rank' is uniform on $\{1, \dots, n\}$, independently of $\textnormal{Rank}_{n-1}(X^{n-1})$.
				Moreover, conditionally on $\textnormal{Rank}_{n-1}(X^{n-1})$, this last rank entirely determines $\textnormal{Rank}_{n}(X^{n})$.
				Hence, the last rank is in bijection with the distribution of $\textnormal{Rank}_{n}(X^{n})$ given $\textnormal{Rank}_{n-1}(X^{n-1})$, so that we may equivalently construct a conditional e-value by constructing an e-value that is valid for the last element of $\textnormal{Rank}_{n}(X^{n})$ being $\textnormal{Unif}(\{1, \dots, n\})$.	
				The above shows what underlies this last-rank result that is popularly used in conformal prediction, and how it generalizes to other settings.
			\end{exm}
			
			\begin{exm}[Sequential sphericity]\label{exm:sphere_inversion_kernel_sequential}
				We now move to the setting where $X^n = (Y_1, \dots, Y_n)$ is a spherical random $n$-vector in $\mathbb{R}^n$.
				That is, it is invariant under the orthogonal group $\mathcal{G}_n$ of $n \times n$ orthonormal matrices under matrix multiplication.
				Let us choose the statistic $H_n(X^n) = X^n / \|X^n\|_2$, which maps from $\mathbb{R}^n$ to the unit sphere in $n$ dimensions.
				Note that, $H_n$ is equivariant: $H_n(GX^n) = GX^n / \|GX^n\|_2 = GX^n / \|X^n\|_2 = GH_n(X^n)$, for any $G \in \mathcal{G}_n$, so that it indeed induces a group action on the unit sphere.
				Under this group action, it has just a single orbit: the unit sphere itself.
				Furthermore, we have the required projection, as we may recover $H_{n-1}(X^{n-1})$ from $H_n(X^n)$ by linearly projecting it onto $\mathbb{R}^{n-1}$ and then dividing by the norm of the resulting vector.
				
				Under sphericity of $X^n$, we have that $H_n(X^n)$ is spherical on the unit hypersphere.
				Conditional on $H_{n-1}(X^{n-1})$, $H_n(X^n)$ is distributed on a semi-unit circle, of points whose first $(n-1)$ coordinates are in the direction of $H_{n-1}(X^{n-1})$.
				Hence, a conditional valid e-value is an e-value that is valid for this distribution on this semi-circle.
				We may generalize this to instead observing, say, two coordinates at each point in time, so that this instead becomes a certain distribution on a semi-sphere.
				
				The standard $t$-test setting is recovered by considering the statistic $\iota_n'H_n(X^n)$ for the unit vector $\iota_n = (1, \dots, 1)/\sqrt{n}$.
				Its distribution and conditional distribution is worked out in detail in Appendix A of \citet{koning2025sequentializing}.
			\end{exm}	
	
	\section{Invariance through statistic}\label{appn:invariance_through}
		\subsection{Examples}
			We illustrate the difference between invariance and invariance through a statistic in two examples.
			In Example \ref{exm:exchangeable}, the random variable is invariant.
			In Example \ref{exm:exchangeable_through_T}, the underlying random variable is not invariant, but looks invariant through certain statistics.
			
			\begin{exm}[Invariance under permutations: exchangeability]\label{exm:exchangeable}
				Suppose we have a single bag and fill it with the numbers 1, 2, 3 and 4.
				We now sample uniformly without replacement from this bag and arrange the numbers in the order they were drawn.
				As each order has the same probability, we say that this outcome is \emph{exchangeable}: invariant under all permutations of the numbers $\{1, 2, 3, 4\}$.
			\end{exm}

			\begin{exm}[Not exchangeable, but exchangeable through a statistic]\label{exm:exchangeable_through_T}
				Suppose we have two bags.
				We fill one with the numbers 1 and 2, and the other with numbers 3 and 4.
				We start by picking a bag with equal probability, and then sequentially draw both numbers from the bag in an exchangeable manner.
				Next, we take the other bag and do the same, after which we arrange the numbers in the order they were drawn.
				
				Here, the resulting order of the numbers is not invariant under all permutations: out of 24 permutations, only the 8 orders 1234, 1243, 2134, 2143, 3412, 3421, 4312 and 4321 can occur.
				The order does \emph{look exchangeable} through the statistic that returns only the first position.
				Indeed, every number is equally likely to land in the first position both under our sampling process, and if we had used an exchangeable sampling process.
				
				The order also looks exchangeable through the statistic $S$ that returns the relative ranks of the first two positions: $S(12\cdot\cdot) = 12$, $S(21\cdot\cdot) = 21$, $S(34\cdot\cdot) = 12$, $S(43\cdot\cdot) = 21$.
				This is because the ranks 12 and 21 happen with equal probability both under our sampling process, and under an exchangeable sampling process.
				
				While we only consider a single orbit in this example --- namely the permutations of 1234 ---  the example extends to multiple orbits.
				Indeed, we may view the numbers 1234 as determined by some preceding sampling process.
			\end{exm}	
			
		\subsection{Counterexample invariance through statistic}\label{appn:counterexample}
			Following Remark \ref{rmk:counterexample}, we discuss a counterexample that shows the condition `$S(\overline{G}Y) \overset{d}{=} S(Y)$', for some test statistic $S$, is insufficient to guarantee that the group invariance test with statistic $S$ is valid.
			In fact, we make a stronger assumption, `$S(Y) \overset{d}{=} S(GY)$ for all $G \in \mathcal{G}$' which implies `$S(\overline{G}Y) \overset{d}{=} S(Y)$' and is also insufficient.
			
			Suppose $Y^n = (Y_1, \dots, Y_n)$ is a random variable on $[0, 1]^n$.
			Let us consider the statistic $S : [0, 1]^n \to [0, 1]$ that returns the first element $S(Y^n) = Y_1$.
			We consider the group $\mathcal{G}$ as the group of permutations on $n$ elements.
			Under this group, the condition $S(Y) \overset{d}{=} S(GY)$ for all $G \in \mathcal{G}$, can be interpreted as equality in distribution of the elements of $Y^n$: $Y_1 \overset{d}{=} Y_i$, for every $i = 1, \dots, n$.
			Moreover, it puts no restriction on the dependence structure of the individual elements, which we will exploit.
			
			Suppose $Y_i \sim \text{Unif}[0, 1]$ for each $i = 1, \dots, n$.
			Now, let us describe the dependence structure: $Y_2 = Y_i$ for all $i \geq 2$, and $Y_1$ and $Y_2$ are exchangeable.
			That is $Y^n = (Y_1, Y_2, \dots, Y_2)$, which is not an exchangeable $n$-vector.
			Given some significance level $\alpha = (k - 1)/n$, $k \in \{1, \dots, n - 1\}$, the classical group invariance test based on the statistic $S$ rejects the hypothesis if $S(Y) = Y_1$ exceeds the $k$th largest value in the the multiset $\{Y_1, Y_2, \dots, Y_2\}$.
			As $Y_1$ is either the largest or smallest value in the multiset, the $k$th largest value must equal $Y_2$.
			By the exchangeability of $Y_1$ and $Y_2$, the probability of rejection equals $\text{Pr}(Y_1 > Y_2) = .5$ regardless of $k$, $n$.
			We can then choose $k \geq 2$ and $n$ such that $(k - 1) / n < .5$, such as $k = 2$ and $n = 3$, to ensure $\alpha < .5$, which in turn means $\text{Pr}(Y_1 > Y_2) > \alpha$, so that the group invariance test is not valid.

	\section{Example: exchangeability}\label{sec:exm_permutation}
		In this section, we discuss a highly concrete toy example of permutations on a small and finite sample space.
		While not as statistically interesting as the examples in Section \ref{sec:gaussian}, it is more tangible as the group itself is finite and easy to understand.
	
		\subsection{Exchangeability on a finite sample space}
			Suppose our sample space $\mathcal{Y}$ consists of the vectors $[1, 2, 3]$, $[1, 1, 2]$ and all their permutations.
			As a group $\mathcal{G}$, we consider the permutations on 3 elements, which we will denote by $\{abc, acb, bac, bca, cab, cba\}$.
			For example, $bac$ represents the permutation that swaps the first two elements.
			
			The orbits are then given by all permutations of $[1,2,3]$ and $[1,1,2]$
			\begin{align*}
				O_{[1,1,2]} = \{[1,1,2], [1,2,1], [2,1,1]\},
			\end{align*}
			and
			\begin{align*}
				O_{[1,2,3]} = \{[1,2,3],[1,3,2],[2,1,3],[2,3,1],[3,1,2],[3,2,1]\}.
			\end{align*}
			As $\mathcal{Y}$-valued orbit representatives, we pick the unique element in the orbit that is sorted in ascending order: $[1,1,2]$ and $[1,2,3]$.
			
			For simplicity, let us restrict ourselves to $O_{[1,2,3]}$ first.
			On this orbit, the inversion kernel $\gamma$ is defined as the unique permutation that brings the element $[1,2,3]$ to $z \in O_{[1,2,3]}$.
			Moreover, on this orbit, the null hypothesis then states that $\gamma(Y)$ is uniform on the permutations, which in this case is equivalent to the hypothesis that $Y$ is uniform on $O_{[1,2,3]}$.
			
			Now let us restrict ourselves to $O_{[1,1,2]}$.
			On this orbit, there are multiple permutations that may bring a given element back to $[1,1,2]$.
			For example, both $bac$, as well as the identity permutation $abc$ bring $[1,1,2]$ to itself, so that the group action is not \emph{free}.
			More generally, any permutation that brings $[1,1,2]$ to $z \in O_{[1,1,2]}$, can be preceded by $bac$, and the result still brings $[1,1,2]$ to $z \in O_{[1,1,2]}$.
			Even more abstractly speaking, let $\mathcal{S}_{[y]} = \{G \in \mathcal{G} : G[y] = [y]\}$ be the stabilizer subgroup of $[y]$ (the subgroup that leaves $[y]$ unchanged).
			Then, if $G^* \in \mathcal{G}$ carries $[y]$ to $y$, so does any element in $G^*\mathcal{S}_{[y]}$.
			
			To construct the inversion kernel $\gamma$ on $O_{[1,1,2]}$, let $\overline{S}_{[y]}$ denote a uniform distribution on $\{abc, bac\}$, which is well-defined as $\mathcal{S}_{[y]}$ is a compact subgroup and so admits a Haar probability measure (see Lemma \ref{lem:stabilizer}).
			Moreover, let $G_y$ be an arbitrary permutation that carries $[y]$ to $y$, say $G_{[1,1,2]} = abc$, $G_{[1,2,1]} = acb$ and $G_{[2,1,1]} = cba$.
			Then, we define the inversion kernel as $\gamma(y) = G_y\overline{S}_{[y]}$.
			Concretely, this means that $\gamma([1,1,2]) \sim \text{Unif}(abc, bac)$, $\gamma([1,2,1]) \sim \text{Unif}(acb, bca)$ and $\gamma([2,1,1]) \sim \text{Unif}(cba, cab)$.
			If $Y$ is indeed uniform on $O_{[1,1,2]}$, then $G_Y$ is uniform on $\{abc, acb, cba\}$ and so $\gamma(Y)$ is uniform on $\mathcal{G}$.
			
			The definition of $\gamma$ on the sample space $\mathcal{Y} = O_{[1,2,3]} \cup O_{[1,1,2]}$ is obtained by combining the definitions on the two separate orbits.

	\section{Omitted proofs}\label{sec:technical}
		
		\subsection{Proof of Theorem \ref{thm:e-values}}\label{proof:thm:e-values}
			\begin{proof}
				We prove (i), as (ii) follows analogously.
				Recall from Lemma \ref{lem:invariance_random_variable} that for a $\mathcal{G}$ invariant random variable $Y$, we have $Y \overset{d}{=} \overline{G}[Y]$, so that $\mathbb{E}_{Y}[\varepsilon(Y)] = \mathbb{E}_{[Y]}\mathbb{E}_{\overline{G}}[\varepsilon(\overline{G}[Y])]$, by Tonelli's theorem.
				Moreover, recall that $\overline{G}y \sim \textnormal{Unif}(O_y)$, for fixed $y$.
				
				For the `$\impliedby$' direction, we may simply take the expectation over $[Y]$ on both sides in the right-hand side of (i) to obtain:
				\begin{align*}
					\mathbb{E}_{Y}[\varepsilon(Y)]
						= \mathbb{E}_{[Y]}\mathbb{E}_{\overline{G}}[\varepsilon(\overline{G}[Y])] 
						= \mathbb{E}_{[Y]}\mathbb{E}^{\textnormal{Unif}(O_{[Y]})}[\varepsilon] 
						\leq \mathbb{E}_{[Y]}[1] = 1.
				\end{align*}
				
				For the `$\implies$' direction, we assume $\mathbb{E}_P[\varepsilon] \leq 1$ for every $\mathcal{G}$ invariant probability $P$.
				Fix an orbit $O \in \mathcal{Y} / \mathcal{G}$.
				An example of a $\mathcal{G}$ invariant probability is $\textnormal{Unif}(O)$.
				Hence, $\mathbb{E}^{\textnormal{Unif}(O)}[\varepsilon] \leq 1$.
				As $O$ is arbitrary, this must hold for every $O$.
			\end{proof}
		
		\subsection{Proof of Theorem \ref{thm:gi_p_value}}\label{proof:theorem_traditional}
			\begin{proof}
				Recall that invariance of $Y$ through $T$ means that $T(\overline{G}Y) \mid O_Y \overset{d}{=} T(Y) \mid O_Y$, $O_Y$-almost surely.
				This means there exists a measurable set $\mathcal{A} \subseteq \mathcal{Y} / \mathcal{G}$ with $\mathbb{P}(O_Y \in \mathcal{A}) = 1$, such that for every orbit $O \in \mathcal{A}$,
				\begin{align*}
					T(\overline{G}Y) \mid (O_Y = O) \overset{d}{=} T(Y) \mid (O_Y = O).
				\end{align*}
				Fix an arbitrary orbit $O \in \mathcal{A}$ and let $[z]$ be its orbit representative.
				We then have,
				\begin{align*}
					T(Y) \mid (O_Y = O)
						\overset{d}{=} T(\overline{G}Y) \mid (O_Y = O)
						\overset{d}{=} T(\overline{G}[z]),
				\end{align*}
				where the second equality-in-distribution follows from $\overline{G}Y \mid (O_Y = O) \overset{d}{=} \overline{G}[z] \sim \textnormal{Unif}(O)$.
				Hence, $q_\alpha^{Y}[T(Y) \mid O_Y = O] = q_\alpha^{\overline{G}}[T(\overline{G}[z])]$ and $\mathbb{E}_Y[\varepsilon_\alpha(Y) \mid O_Y = O] = \mathbb{E}_{\overline{G}}[\varepsilon_\alpha(\overline{G}[z])]$.
				Recall
				\begin{align*}
					\varepsilon_\alpha(y)
						= \frac{1}{\alpha} \mathbb{I}\left\{T(y) > q_\alpha^{\overline{G}}[T(\overline{G}y)]\right\} + \frac{c([y])}{\alpha}\mathbb{I}\left\{T(y) = q_\alpha^{\overline{G}}[T(\overline{G}y)]\right\}.
				\end{align*}
				Let $\overline{G}^* \sim \textnormal{Unif}(\mathcal{G})$, independently from $\overline{G}$.
				Then,
				\begin{align*}
					&\mathbb{E}_{\overline{G}^*}[\varepsilon_\alpha(\overline{G}^*[z])] \\
						&= \mathbb{E}_{\overline{G}^*}\left[\frac{1}{\alpha} \mathbb{I}\left\{T(\overline{G}^*[z]) > q_\alpha^{\overline{G}}[T(\overline{G}\overline{G}^*[z])]\right\} + \frac{c([\overline{G}^*[z]])}{\alpha}\mathbb{I}\left\{T(\overline{G}^*[z]) = q_\alpha^{\overline{G}}[T(\overline{G}\overline{G}^*[z])]\right\}\right] \\
						&= \mathbb{E}_{\overline{G}^*}\left[\frac{1}{\alpha} \mathbb{I}\left\{T(\overline{G}^*[z]) > q_\alpha^{\overline{G}}[T(\overline{G}[z])]\right\} + \frac{c([z])}{\alpha}\mathbb{I}\left\{T(\overline{G}^*[z]) = q_\alpha^{\overline{G}}[T(\overline{G}[z])]\right\}\right] \\
						&= \frac{1}{\alpha}\mathbb{P}_{\overline{G}^*}\left(T(\overline{G}^*[z]) > q_\alpha^{\overline{G}}[T(\overline{G}[z])]\right) + \frac{c([z])}{\alpha}\mathbb{P}_{\overline{G}^*}\left(T(\overline{G}^*[z]) = q_\alpha^{\overline{G}}[T(\overline{G}[z])]\right),
				\end{align*}
				where the second equality follows from the fact that $\overline{G}\overline{G}^* \sim \textnormal{Unif}(\mathcal{G})$, and $[Gz] = [z]$ for all $G \in \mathcal{G}$.
				Choose
				\begin{align*}
					c([z]) 
						= \frac{\alpha - \mathbb{P}_{\overline{G}^*}\left(T(\overline{G}^*[z]) > q_{\alpha}^{\overline{G}}\left[T(\overline{G}[z])\right]\right)}{\mathbb{P}_{\overline{G}^*}\left(T(\overline{G}^*[z]) = q_\alpha^{\overline{G}}[T(\overline{G}[z])]\right)},
				\end{align*}
				if $\mathbb{P}_{\overline{G}^*}\left(T(\overline{G}^*[z]) = q_\alpha^{\overline{G}}[T(\overline{G}[z])]\right) > 0$ and $c([z]) = 0$, otherwise.
				This yields $\mathbb{E}_{\overline{G}}[\varepsilon_\alpha(\overline{G}[z])] = 1$, and so $\mathbb{E}_{Y}[\varepsilon_\alpha(Y) \mid O_Y = O] = 1$.
				Since $O \in \mathcal{A}$ was arbitrary, this holds for every $O \in \mathcal{A}$, and so $\mathbb{E}_{Y}[\varepsilon_\alpha(Y) \mid O_Y] = 1$, almost surely.
			\end{proof}
			
		\subsection{Proof of Theorem \ref{thm:essentially_complete}}\label{proof:essentially_complete}
			\begin{proof}
				The 	`$\impliedby$' direction follows from
				\begin{align*}
					\mathbb{E}_{\overline{G}}[\varepsilon_T(\overline{G}y)]
						= \mathbb{E}_{\overline{G}_1}\left[\frac{T(\overline{G}_1y)}{\mathbb{E}_{\overline{G}_2}T(\overline{G}_2\overline{G}_1y)}\right]
						= \mathbb{E}_{\overline{G}_1}\left[\frac{T(\overline{G}_1y)}{\mathbb{E}_{\overline{G}_2}T(\overline{G}_2y)}\right]
						= \frac{\mathbb{E}_{\overline{G}_1}T(\overline{G}_1y)}{\mathbb{E}_{\overline{G}_2}T(\overline{G}_2y)}
						= 1,
				\end{align*}
				and applying Theorem \ref{thm:e-values}.
				For the `$\implies$' direction, assume $\varepsilon$ is some exact e-value for $\mathcal{G}$ invariance.
				Choose $T = \varepsilon$, so that
				\begin{align*}
					\varepsilon_\varepsilon(y)
						= \frac{\varepsilon(y)}{\mathbb{E}_{\overline{G}} \varepsilon(\overline{G}y)}
						= \varepsilon(y),
				\end{align*}
				where the final equality follows from the fact that $\mathbb{E}_{\overline{G}} \varepsilon(\overline{G}y) = 1$, by Theorem \ref{thm:e-values}.
			\end{proof}	
		
		\subsection{Proof of Proposition \ref{prp:e-value_invariance_through}}\label{proof:e-value_invariance_through}
			\begin{proof}
				The assumption $\mathbb{E}_Y[T(Y) \mid O_Y] = \mathbb{E}[T(\overline{G}Y) \mid O_Y]$, almost surely means that there exists a subset $\mathcal{A} \subseteq \mathcal{Y} / \mathcal{G}$, with $\mathbb{P}_Y(O_Y \in \mathcal{A}) = 1$, such that for every $O \in \mathcal{A}$,
				\begin{align*}
					\mathbb{E}_Y[T(Y) \mid O_Y = O] 
						= \mathbb{E}_Y[T(\overline{G}Y) \mid O_Y = O].
				\end{align*}
				Fix an arbitrary orbit $O \in \mathcal{A}$ and let $[z]$ be its orbit representative.
				We then have,
				\begin{align*}
					\mathbb{E}_Y[T(Y) \mid O_Y = O]
						= \mathbb{E}_{\overline{G}, Y}[T(\overline{G}Y) \mid O_Y = O]
						= \mathbb{E}_{\overline{G}}[T(\overline{G}[z])],
				\end{align*}
				where the second equality follows from $\overline{G}Y \mid O_Y = O \overset{d}{=} \overline{G}[z] \sim \textnormal{Unif}(O)$.
				As a consequence,
				\begin{align*}
					\mathbb{E}_Y[\varepsilon_T(Y) \mid O_Y = O]
						&= \mathbb{E}_Y\left[\frac{T(Y)}{\mathbb{E}_{\overline{G}}[T(\overline{G}Y)]} \middle| O_Y = O\right]
						= \frac{\mathbb{E}_Y[T(Y) \mid O_Y = O]}{\mathbb{E}_{\overline{G}}[T(\overline{G}[z])]} = 1,
				\end{align*}
				where the second equality uses that $y \mapsto \mathbb{E}_{\overline{G}}[T(\overline{G}y)]$ is constant on each orbit.
				Now as $O$ was arbitrarily chosen, this holds for every $O \in \mathcal{A}$, and so $\mathbb{E}_Y[\varepsilon_T(Y) \mid O_Y] = 1$, almost surely.
			\end{proof}
						
		\subsection{Proof of Theorem \ref{thm:monte_carlo}}\label{appn:proof_monte_carlo}
			
			\begin{proof}
				It suffices to show that $\mathbb{E}_{\overline{G}^{(1)}, \dots, \overline{G}^{(M)}} \mathbb{E}_{\overline{G}}[\varepsilon_T^M(\overline{G}z)] = 1$, for some $z$ on each orbit $O$.
				The strategy is to first show that the tuple $(T(\overline{G}z)$, $T(\overline{G}^{(1)}z)$, $\dots$, $T(\overline{G}^{(M)}z))$ is exchangeable, and to then apply Proposition \ref{prp:e-value_invariance_through}.
				
				Fix some arbitrary orbit $O \in \mathcal{Y} / \mathcal{G}$ and some element $z \in O$.
				Note that
				\begin{align*}
					\varepsilon_T^M(\overline{G}z) 
						= \frac{T(\overline{G}z)}{\frac{1}{M + 1} \sum_{i = 0}^M T(\overline{G}^{(i)} \overline{G}z)}.
				\end{align*}
				Central to this object is the tuple $(\overline{G}, \overline{G}^{(1)}\overline{G}, \dots, \overline{G}^{(M)}\overline{G})$.
				
				Note that $(\overline{G}^{(1)}\overline{G}, \dots, \overline{G}^{(M)}\overline{G})$ is independent from $\overline{G}$, as
				\begin{align*}
					(\overline{G}^{(1)}\overline{G}, \dots, \overline{G}^{(M)}\overline{G}) \mid (\overline{G} = g)
						= (\overline{G}^{(1)}g, \dots, \overline{G}^{(M)}g)
						\overset{d}{=} (\overline{G}^{(1)}, \dots, \overline{G}^{(M)}).
				\end{align*}
				As the elements of the tuple $(\overline{G}^{(1)}, \dots, \overline{G}^{(M)})$ are mutually independent, we have that the tuple $(\overline{G}, \overline{G}^{(1)}\overline{G}, \dots, \overline{G}^{(M)}\overline{G})$ is mutually independent.
				Moreover, each element is marginally $\textnormal{Unif}(\mathcal{G})$, so that the tuple is i.i.d. and hence exchangeable.

				We now prepare some notation to show that we may apply Proposition \ref{prp:e-value_invariance_through}.
				Let us write $(\overline{G}_*^{(0)}, \dots, \overline{G}_*^{(M)}) = (\overline{G}, \overline{G}^{(1)}\overline{G}, \dots, \overline{G}^{(M)}\overline{G})$, so that
				\begin{align*}
					\mathbb{E}_{\overline{G}^{(1)}, \dots, \overline{G}^{(M)}}\left[\mathbb{E}_{\overline{G}}\left[\frac{T(\overline{G}z)}{\frac{1}{M + 1} \sum_{i = 0}^M T(\overline{G}^{(i)} \overline{G}z)}\right]\right]
						= \mathbb{E}_{\overline{G}_*^{(0)}, \overline{G}_*^{(1)}, \dots, \overline{G}_*^{(M)}}\left[\frac{T(\overline{G}_*^{(0)}z)}{\frac{1}{M + 1} \sum_{i = 0}^M T(\overline{G}_*^{(i)}z)}\right].
				\end{align*}
				Define the tuple
				\begin{align*}
					\mathcal{T} := (T(\overline{G}_*^{(0)}z), T(\overline{G}_*^{(1)}z), \dots, T(\overline{G}_*^{(M)}z)),
				\end{align*}
				which is exchangeable as $(\overline{G}_*^{(0)}, \dots, \overline{G}_*^{(M)})$ is exchangeable.
				Moreover, define the statistic $S$ that returns the first element of such a tuple $\mathcal{T}$.
				Finally, let $\overline{P}$ denote a uniform permutation on tuples of $(M+1)$ elements.
				Then, by Proposition \ref{prp:e-value_invariance_through},
				\begin{align*}
					\mathbb{E}_{\overline{G}_*^{(0)}, \overline{G}_*^{(1)}, \dots, \overline{G}_*^{(M)}}\left[\frac{T(\overline{G}_*^{(0)}z)}{\frac{1}{M + 1} \sum_{i = 0}^M T(\overline{G}_*^{(i)}z)}\right]
						= \mathbb{E}_{\mathcal{T}}\left[\frac{S(\mathcal{T})}{\mathbb{E}_{\overline{P}}[S(\overline{P}\mathcal{T})]}\right]
						= 1.
				\end{align*}
			\end{proof}

		\subsection{Proof of Lemma \ref{lem:U-optimal}} \label{proof:U-optimal}
			
			\begin{proof}
				For every $\varepsilon \in \mathcal{E}_1(\mathbb{P})$,
				\begin{align*}
					\frac{d\mathbb{Q}_a}{d\mathbb{P}}(y) U(\varepsilon_\lambda(y)) - \lambda \varepsilon_\lambda(y)
						\geq \frac{d\mathbb{Q}_a}{d\mathbb{P}}(y)U(\varepsilon(y)) - \lambda \varepsilon(y), \quad \mathbb{P}\textnormal{-a.s.}
				\end{align*}
				Integrating both sides with respect to $\mathbb{P}$ and applying the definition of the RN-derivatives:
				\begin{align*}
					\mathbb{E}^{\mathbb{Q}_a}[U(\varepsilon_\lambda)] - \lambda \mathbb{E}^{\mathbb{P}}[\varepsilon_\lambda]
						\geq \mathbb{E}^{\mathbb{Q}_a}[U(\varepsilon)] - \lambda \mathbb{E}^{\mathbb{P}}[\varepsilon].
				\end{align*}
				Since $\varepsilon_\lambda, \varepsilon \in \mathcal{E}_1(\mathbb{P})$ the second terms cancel.
				This yields $\mathbb{E}^{\mathbb{Q}_a}[U(\varepsilon_\lambda)] \geq \mathbb{E}^{\mathbb{Q}_a}[U(\varepsilon)]$, for all $\varepsilon \in \mathcal{E}_1(\mathbb{P})$ for which the right-hand-side is well defined.
			\end{proof}
				
		\subsection{Proof of Theorem \ref{thm:characterize_optimal}}\label{proof:characterize_optimal}
			\begin{proof}
				The first claim follows from Theorem~\ref{thm:e-values}.
				For the second claim, suppose that $\varepsilon' \in F_+^\mathcal{Y}$ is some other e-value that is valid for $\mathcal{G}$ invariance. 
				Then, for each $O \in \mathcal{Y} / \mathcal{G}$, $\varepsilon'_{|O}$ is also a valid e-value for $\mathrm{Unif}(O)$ by Theorem~\ref{thm:e-values}. 
				Hence, by the assumption on $\varepsilon^*$, 
				\begin{align*}
					K_O\left(\varepsilon_{|O}^*\right) \geq K_O\left(\varepsilon'_{|O}\right), \quad\text{for every } O \in \mathcal{Y} / \mathcal{G}.
				\end{align*}
				Since $\Psi$ is non-decreasing in each of its inputs, it follows that
				\begin{align*}
					K(\varepsilon^*) 
						\equiv \Psi\left(\bigl(K_O\bigl(\varepsilon_{|O}^*\bigr)\bigr)_{O \in \mathcal{Y} / \mathcal{G}}\right)
						\geq \Psi\left(\bigl(K_O\bigl(\varepsilon'_{|O}\bigr)\bigr)_{O \in \mathcal{Y} / \mathcal{G}}\right)
						\equiv K(\varepsilon'),
				\end{align*}
				which proves the second claim.
			\end{proof}

		\subsection{Proof of Lemma \ref{lem:conditional_unconditional_RN}}\label{proof:conditional_unconditional_RN}
			
			\begin{proof}
				We start by showing the existence of the subprobability kernel.
				
				Let $\mu_a := \mathbb{Q}_a \circ \pi^{-1}$.
				Since $\mathbb{Q}_a \ll \overline{\mathbb{Q}}$, we have $\mu_a \ll \overline{\mathbb{Q}} \circ \pi^{-1} = \mu$.
				Let $\mathbb{Q}_a'(\cdot \mid O)$ be a version of the conditional law of $\mathbb{Q}_a$ given $O$, so that $\mathbb{Q}_a(\cdot) = \int \mathbb{Q}_a'(\cdot \mid O) d\mu_a(O)$.
				Define
				\begin{align*}
					\mathbb{Q}_a(\cdot \mid O) := \frac{d\mu_a}{d\mu} \mathbb{Q}_a'(\cdot \mid O).
				\end{align*}
				Then, we indeed have the desired
				\begin{align*}
					\int \mathbb{Q}_a(\cdot \mid O) d\mu(O)
						= \int \frac{d\mu_a}{d\mu} \mathbb{Q}_a'(\cdot \mid O) d\mu(O)
						= \int \mathbb{Q}_a'(\cdot \mid O) d\mu_a(O)
						= \mathbb{Q}_a(\cdot).
				\end{align*}
				Moreover, since $\mathbb{Q}_a \ll \overline{\mathbb{Q}}$ and $\textnormal{Unif}(O)$ is a version of $\overline{\mathbb{Q}}(\cdot \mid O)$, it follows that $\mathbb{Q}_a'(\cdot \mid O) \ll \textnormal{Unif}(O)$ $\mu_a$-almost surely and so $\mathbb{Q}_a(\cdot \mid O) \ll \textnormal{Unif}(O)$ $\mu$-almost surely.
	
				Now, using that $\textnormal{Unif}(O)$ is a version of $\overline{\mathbb{Q}}(\cdot \mid O)$, we have for every event $A$,
				\begin{align*}
					\int 1_A(y) d\mathbb{Q}_a(y)
						&= \int \int 1_A(y) d\mathbb{Q}_a(y \mid O) d\mu(O) \\
						&= \int \int 1_A(y) \frac{d\mathbb{Q}_a(\cdot \mid O)}{d\textnormal{Unif}(O)}(y) d\textnormal{Unif}(O)(y) d\mu(O) \\
						&= \int 1_A(y) \frac{d\mathbb{Q}_a(\cdot \mid O_y)}{d\textnormal{Unif}(O_y)}(y) d\overline{\mathbb{Q}}(y).
				\end{align*}
				Hence, $y \mapsto \frac{d\mathbb{Q}_a(\cdot \mid O_y)}{d\textnormal{Unif}(O_y)}(y)$ is a version of $\frac{d\mathbb{Q}_a}{d\overline{\mathbb{Q}}}$.
			\end{proof}

		\subsection{Proof of Theorem \ref{thm:utility-optimal-invariant}}\label{proof:utility-optimal-invariant}
			
			\begin{proof}
				Fix an orbit $O$.
				By Lemma \ref{lem:conditional_unconditional_RN}, $d\mathbb{Q}_a(\cdot \mid O) / d\textnormal{Unif}(O)$ is a version of $d\mathbb{Q}_a/d\overline{\mathbb{Q}}$ on $O$.
				Hence, assumption (ii) can be written as
				\begin{align*}
					\varepsilon^U(y) \in \textnormal{argmax}_{x \in [0, \infty]} \frac{d\mathbb{Q}_a(\cdot \mid O)}{d\textnormal{Unif}(O)}(y) U(x) - \lambda_O x, \quad \textnormal{Unif}(O)\textnormal{-a.s.}
				\end{align*}
				Combining this with (i) and Lemma \ref{lem:U-optimal}, choosing $\mathbb{P} = \textnormal{Unif}(O)$ and $\mathbb{Q}_a = \mathbb{Q}_a(\cdot \mid O)$, yields that $\varepsilon^U$ is $U$-optimal among e-values that are exact for $\textnormal{Unif}(O)$.
				Applying Theorem \ref{thm:characterize_optimal} upgrades this to $\mathcal{G}$ validity and global $U$-optimality of $\varepsilon^U$.
			\end{proof}
								
		\subsection{Proof of Theorem \ref{thm:infimum_of_martingales}}\label{proof:infimum_of_martingales}
			\begin{proof}
				We start with the right-to-left direction.
				For every $O$, let $(\varepsilon_n^O)_{n \geq 0}$ be a non-negative martingale starting at 1 that bounds $(\varepsilon_n)_{n \geq 0}$ from above, $\textnormal{Unif}(O)$-a.s.
				Then for every stopping time $\tau$,
				\begin{align*}
					\mathbb{E}^{\textnormal{Unif}(O)}[\varepsilon_\tau]
						&\leq  \mathbb{E}^{\textnormal{Unif}(O)}[\varepsilon_\tau^{O}] 
						= 1,
				\end{align*}
				where the second inequality from Doob's optional stopping theorem for non-negative martingales, and the assumption that the martingale starts at 1.
				By Theorem \ref{thm:e-process}, $(\varepsilon_n)_{n \geq 0}$ is anytime valid for $\mathcal{G}$ invariance.
				
				For the converse direction, suppose that $(\varepsilon_n)_{n \geq 0}$ is anytime valid for $\mathcal{G}$ invariance.
				By Theorem \ref{thm:e-process}, for every orbit $O$ and stopping time $\tau$,
				\begin{align*}
					\mathbb{E}^{\textnormal{Unif}(O)}[\varepsilon_\tau] \leq 1.
				\end{align*}
				Fix an orbit $O$ and apply Lemma 6 in \citet{ramdas2022admissible} with the simple hypothesis $\{\textnormal{Unif}(O)\}$.
				This yields the existence of a non-negative $\textnormal{Unif}(O)$-martingale $(\varepsilon_n^O)_{n \geq 0}$ starting at $1$ that bounds $(\varepsilon_n)_{n \geq 0}$ from above $\textnormal{Unif}(O)$-a.s. 
				Since the orbit $O$ was arbitrarily chosen, such a martingale exists for every orbit.
			\end{proof}
		
		\subsection{Proof of Lemma \ref{lem:adapted_distribution}}\label{proof:adapted_distribution}	
			To prove Lemma \ref{lem:adapted_distribution}, we prove a more general result in Proposition \ref{prp:conditional_invariance}.
			Lemma \ref{lem:adapted_distribution} is recovered by choosing $h$ equal to the relevant projection map.
		
			Let $\mathcal{Y}$ be our sample space on which our group $\mathcal{G}$ acts.
			Let $\mathcal{Z}$ be some other space.
			Suppose $h : \mathcal{Y} \to \mathcal{Z}$ is continuous.
			Assume $h$ induces a group action on $\mathcal{Z}$.
			That is, we assume $h(y_1) = h(y_2) \implies h(Gy_1) = h(Gy_2)$, for all $G \in \mathcal{G}$ and $y_1, y_2 \in \mathcal{Y}$.
			This means $h$ is equivariant for this group action on $\mathcal{Z}$: $h(Gy) = Gh(y)$.
			
			Our goal is to characterize the conditional distribution of $Y \mid (O_Y, h(Y))$ by a subgroup of $\mathcal{G}$. 
			We start by characterizing the subgroup, and showing that it is compact.
			
			Let us consider the subset $\mathcal{K}^h$ of $\mathcal{G}$ that stabilizes the statistic $h$ of the data:
			\begin{align*}
				\mathcal{K}^h(y) 
					= \{G \in \mathcal{G} : h(Gy) = h(y)\}
					= \{G \in \mathcal{G} : Gh(y) = h(y)\},
			\end{align*}
			Such a set $\mathcal{K}^h(y)$ is also known as a stabilizer subgroup.
			The fact that it is indeed a subgroup, and crucially its compactness are captured in Lemma \ref{lem:stabilizer}.
			
			\begin{lem}\label{lem:stabilizer}
				$\mathcal{K}^h(y)$ is a compact subgroup of $\mathcal{G}$.
			\end{lem}
			\begin{proof}
				We start by showing that $\mathcal{K}^h(y)$ is a subgroup.
				First, the identity $I$ is trivially in $\mathcal{K}^h(y)$.
				For any $K_1, K_2 \in \mathcal{K}^h(y)$, it is closed under composition: $K_1K_2h(y) = K_1 h(y) = h(y)$.
				Moreover, for any $K \in \mathcal{K}^h(y)$, it contains its inverse $K^{-1}$: $h(y) = Ih(y) = K^{-1}Kh(y) = K^{-1}h(y)$.
				
				Next, we show that $\mathcal{K}^h(y)$ is topologically closed.
				Define the map $f_y : \mathcal{G} \to \mathcal{Z}$ as the composition between $h$ and the group action: $f_y(G) = h(Gy)$.
				As both $h$ and the group action are continuous, their composition $f_y$ is also continuous.
				Since we latently assume any space we consider is Hausdorff, $\mathcal{Z}$ is also a $T_1$ space, so that $\{h(y)\}$ is closed.
				Hence, $\mathcal{K}^h(y)$ is the pre-image of the closed set $\{h(y)\}$ under a continuous map, and so $\mathcal{K}^h(y)$ is also closed.
				As $\mathcal{K}^h(y)$ is a closed subset of the compact set $\mathcal{G}$, it is also compact.
			\end{proof}
	
			In Proposition \ref{prp:conditional_invariance}, we use this subgroup to characterize the conditional distribution $Y \mid (h(Y), O_Y)$.
			
			\begin{prp}\label{prp:conditional_invariance}
				Let $Y$ be $\mathcal{G}$ invariant and $h : \mathcal{Y} \to \mathcal{Z}$ be $\mathcal{G}$ equivariant.
				For any orbit $O \in \mathcal{Y} / \mathcal{G}$ and $z \in h(O)$, pick $x \in O$ with $h(x) = z$.
				Let $\overline{K}^h \sim \textnormal{Unif}(\mathcal{K}^h(x))$, independent of $Y$.
				Then,
				\begin{align*}
					Y \mid (O_Y = O, h(Y) = z) \overset{d}{=} \overline{K}^hx.
				\end{align*}
			\end{prp}
			\begin{proof}
				We start by characterizing the orbit of $x$ under $\mathcal{K}^h(x)$,
				\begin{align*}
					\mathcal{K}^h(x) x := \{Kx : K \in \mathcal{K}^h(x)\} = \{Gx : G \in \mathcal{G} ,h(Gx)=z\}  = \{y \in O : h(y) = z\}.
				\end{align*}	
				Hence, conditioning on $(O_Y, h(Y)) = (O, z)$ confines $Y$ to this orbit.
				Now, as $Y$ is $\mathcal{G}$ invariant, it is also invariant under any subgroup, including $\mathcal{K}^h(x)$.
				Moreover, for any $K \in \mathcal{K}^h(x)$ we have $O_{KY} = O_Y$ and $h(KY) = h(Y)$, so the event $\{O_Y = O, h(Y) = z\}$ is $\mathcal{K}^h(x)$ invariant.
				As a consequence, $Y \mid (O_Y = O, h(Y) = z)$ is $\mathcal{K}^h(x)$ invariant on $\{y \in O : h(y) = z\}$ and so $Y \mid (O_Y = O, h(Y) = z) \overset{d}{=} \overline{K}^hx$.
			\end{proof}

			\stitle{}
			\sdescription{}
		
	\section{Full table hot hand application}\label{appn:hot_hand}
		\begin{table}[!htb]
			\centering
			\begin{tabular}{lrrrrrr}
			\toprule
			\multicolumn{1}{c}{Trigger} & \multicolumn{2}{c}{1 hit} & \multicolumn{2}{c}{2 hits} & \multicolumn{2}{c}{3 hits} \\
			\cmidrule(l{3pt}r{3pt}){2-3} \cmidrule(l{3pt}r{3pt}){4-5} \cmidrule(l{3pt}r{3pt}){6-7}
			\multicolumn{1}{c}{ } & \multicolumn{2}{c}{$\beta$} & \multicolumn{2}{c}{$\beta$} & \multicolumn{2}{c}{$\beta$} \\
			\cmidrule(l{3pt}r{3pt}){2-3} \cmidrule(l{3pt}r{3pt}){4-5} \cmidrule(l{3pt}r{3pt}){6-7}
			Shooter ID & 0.85 & 0.90 & 0.85 & 0.90 & 0.85 & 0.90\\
			\midrule
			101 & 0.163 & 0.323 & 0.409 & 0.572 & 0.674 & 0.782\\
			102 & 1.040 & 1.089 & 0.732 & 0.832 & 0.758 & 0.838\\
			103 & 2.737 & 2.068 & 1.582 & 1.414 & 1.316 & 1.232\\
			104 & 0.949 & 1.025 & 0.627 & 0.753 & 0.998 & 1.004\\
			105 & 0.647 & 0.804 & 0.990 & 1.018 & 0.898 & 0.941\\
			\addlinespace
			106 & 4.695 & 2.962 & 2.543 & 1.934 & 2.356 & 1.807\\
			107 & 5.765 & 3.346 & 3.105 & 2.184 & 2.230 & 1.732\\
			108 & 1.040 & 1.065 & 1.675 & 1.426 & 1.284 & 1.191\\
			109 & 2.338 & 1.840 & 3.100 & 2.176 & 3.181 & 2.195\\
			110 & 0.382 & 0.565 & 0.675 & 0.799 & 0.735 & 0.834\\
			\addlinespace
			111 & 1.318 & 1.284 & 1.529 & 1.378 & 1.409 & 1.286\\
			112 & 0.490 & 0.667 & 0.621 & 0.751 & 0.849 & 0.907\\
			113 & 0.242 & 0.418 & 0.391 & 0.559 & 0.509 & 0.655\\
			114 & 1.427 & 1.358 & 1.187 & 1.167 & 1.169 & 1.136\\
			201 & 0.613 & 0.779 & 0.924 & 0.979 & 0.764 & 0.850\\
			\addlinespace
			202 & 1.938 & 1.636 & 1.090 & 1.085 & 1.099 & 1.073\\
			203 & 3.076 & 2.227 & 1.156 & 1.135 & 1.201 & 1.142\\
			204 & 0.548 & 0.711 & 0.909 & 0.954 & 0.971 & 0.986\\
			205 & 0.441 & 0.616 & 1.001 & 1.018 & 0.725 & 0.816\\
			206 & 0.323 & 0.510 & 0.758 & 0.855 & 0.734 & 0.825\\
			\addlinespace
			207 & 2.503 & 1.950 & 4.173 & 2.636 & 2.405 & 1.815\\
			208 & 0.233 & 0.409 & 0.679 & 0.798 & 1.279 & 1.192\\
			209 & 0.428 & 0.612 & 1.053 & 1.062 & 1.109 & 1.084\\
			210 & 0.306 & 0.487 & 1.330 & 1.234 & 1.375 & 1.251\\
			211 & 0.422 & 0.602 & 0.423 & 0.587 & 0.453 & 0.603\\
			\addlinespace
			212 & 0.452 & 0.620 & 0.643 & 0.755 & 1.000 & 1.000\\
			\addlinespace
			Product e-value & 0.007 & 0.180 & 3.108 & 4.460 & 7.489 & 5.525\\
			\bottomrule
			\end{tabular}
			\caption{Log-optimal e-values for each shooter in the controlled shooting experiment of \citet{gilovich1985hot} for exchangeability against several hot hand alternatives, triggering after 1-3 hits for a modest effect $(\beta = 0.85)$ and weak effect $(\beta = 0.9)$. The final row reports the product e-value of the corresponding column.}
			\label{tab:hot_hand_full}
		\end{table}
			
\end{document}